%% file: qbkix3d-docs.tex
\documentclass[10pt,sort&compress]{elsarticle}
\pdfoutput=1

\input{preambles.tex}

\begin{document}

\input{topmatter}
\input{intro}
\input{formulation}

\input{algo}

\input{error}

\input{complexity}
\input{results}

\input{conclusion}

\input{appendix}

\bibliographystyle{alpha}

\bibliography{refs}

\end{document}

%% file: topmatter.tex
\title{A robust solver for elliptic PDEs in 3D complex geometries}

\author[nyu]{Matthew J. Morse} \ead{mmorse@cs.nyu.edu}
\author[cub]{Abtin Rahimian} \ead{arahimian@acm.org}
\author[nyu]{Denis Zorin} \ead{dzorin@cs.nyu.edu}
\address[nyu]{Courant Institute of Mathematical Sciences, New York
  University, New York, NY 10003}
\address[cub]{Department of Computer Science, University of Colorado - Boulder, Boulder, CO 80309}
\begin{abstract}
  We develop a boundary integral equation solver for elliptic partial differential equations on complex \threed geometries.
    Our method is efficient, high-order accurate and robustly handles complex geometries.
  A key component is our singular and near-singular layer potential evaluation scheme, \qbkix: a simple extrapolation of the solution along a line to the boundary.
  We present a series of geometry-processing algorithms required for \qbkix to run efficiently with accuracy guarantees on arbitrary geometries and an adaptive upsampling scheme based on a iteration-free heuristic for quadrature error.
  We validate the accuracy and performance with a series of numerical tests and compare our approach to a competing local evaluation method.
\end{abstract}
\maketitle

%% file: intro.tex
\section{Introduction\label{sec:intro}}
Linear elliptic homogeneous partial differential equations (\pdes) play an important role in modeling many physical interactions, including electrostatics, elastostatics, acoustic scattering, and viscous fluid flow. 
Ideas from potential theory allow us to reformulate the associated boundary value problem (\bvp) as an integral equation \cite{hsiao2008boundary}. 
The solution to the \bvp can then be expressed as a surface convolution against the \pde's fundamental solution called a \textit{layer potential}. 
Discretizing this boundary integral equation (\bie) formulation offers several advantages over commonly used \pde discretization methods such as finite element or finite volume methods.

First, the system of equations uses asymptotically fewer variables because only the boundary of the \pde's domain requires discretization.
There is no need to directly discretize the domain itself, which is often time-consuming and error-prone, especially when complex or unbounded domains are involved. 
This makes the boundary integral formulation well-suited for electromagnetic problems \cite{nedelec2001acoustic} and indispensable for particulate flow simulations with changing, moving, or deforming geometries \cite{pozrikidis-1992-BISMLVF}).
Second, although the algebraic system resulting from discretization of \bie's is dense, efficient methods based on the Fast Multipole Method \cite{greengard1987fast} can solve it in $O(N)$ time.
A suitable integral formulation can yield a well-conditioned system that can be solved using an iterative method like \gmres in relatively few iterations.
Third, high-order quadrature rules can be leveraged to dramatically improve the accuracy of a given discretization size.

For elliptic problems with smooth domain boundaries, fast, high-order methods have a significant advantage over standard methods, drastically reducing the number of degrees of freedom needed to approximate a solution to a given accuracy. 
However, achieving this with a \bie discretization presents a significant challenge.
In particular, integral equation solvers require accurate quadrature rules for \emph{singular} integrals, as
the formulation requires the solution of an integral equation involving the singular fundamental solution of the \pde.
Moreover, if the solution needs to be evaluated arbitrarily close to the boundary, then one must numerically compute \emph{nearly singular} integrals with high-order accuracy (e.g., \cite{bruno2001fast,YBZ,KBGN}).
Precomputing high-order singular/near-singular quadrature weights also presents a considerable problem. 
Such weights necessarily depend on the surface geometry, so each sample point requires a unique set of weights.
Furthermore, the sampling density required for accurate singular/near-singular integration is highly dependent on the boundary geometry.
For example, two nearly touching pieces of the boundary require a sampling density proportional to the distance between them. 
Applying such a fine discretization globally would be prohibitively expensive, highlighting the need for adaptive refinement.

%

\subsection{Contributions}
Our main contribution is a high-order, boundary integral solver for non-oscillatory elliptic \pdes, and experimental evaluation of this solver. 
An earlier parallel version of this method is used in \cite{lu2019scalable} to simulate red blood cell flows through complex blood vessel with high numerical accuracy. 
More specifically, the main features of our solver include: 
\begin{itemize}
  \item {\bf Singular and near-singular quadrature scheme.}
  We introduce an approximation-based singular/near-singular quadrature scheme for single- and double-layer potentials in \threed: after computing the solution at a set of nearby \textit{check points}, placed along a line intersecting the target, we extrapolate the solution to the target point. 
  We have named this scheme \qbkix, for reasons that are apparent from \Cref{fig:qbkix-schematic}.
  In order to ensure accuracy of the scheme for complex geometries, a key component of our scheme is a set of geometric criteria for surface sampling needed for accurate integration.
  
  Our approach is motivated by the near-singular evaluation scheme of \cite{YBZ,QB}, which implements a similar scheme that includes an additional on-surface singular evaluation to allow for interpolation of the solution. 
        We eliminate the need for explicit on-surface singular evaluation.
  An important consequence of this include the use of \emph{smooth} quadrature rules only, removing the need for an explicit singular quadrature scheme. 
  This allows for much greater flexibility in the choice of surface representation (e.g., the representation of \cite{YBZ} was explicitly designed to support singular quadratures). 

\item {\bf Surface representation.} 
    Our quadrature scheme enables us to use standard B\'ezier patches to define the domain boundary, which simplifies the use of the solver on \abbrev{CAD} geometry, increases the efficiency of surface evaluation and simplifies parallelization. 
It also allows for adaptive \emph{quad-trees} of patches to approximate complex surfaces with nonuniform curvature distribution efficiently.
        Our method can be applied to other surface representations with minimal changes. 


\item{\bf Refinement for geometric admissibilty and quadrature accuracy.}
An essential aspect of our method is a set of fast adaptive geometry refinement algorithms to ensure that the assumptions required for the validity and accuracy of \qbkix are satisfied.
These conditions are similar in spirit to \cite{RKO} and \cite{wala20193d}, but adapted to the geometry of our particular quadrature scheme.
To guarantee quadrature accuracy of our method, we detail an adaptive $h$-refinement approach for the integral equation discretization points. 
\end{itemize}

We evaluate \qbkix for a variety of problems on complex geometries to demonstrate high-order convergence and compare to \cite{YBZ}.  

\subsection{Related Work\label{sec:related_work}}
We  restrict our discussion to elliptic \pde solvers in \threed using boundary integral formulations. 
The common schemes to discretize boundary integral equations are the \textit{Galerkin} method, the \textit{collocation} method, and the \textit{\nystrom} method \cite{atkinson2009numerical}.  Galerkin and collocation methods are usually referred as Boundary Element Methods (\bem).
\bem has been applied to a variety of problems in elastodynamics, electromagnetics and acoustics \cite{chaillat2017fast,chaillat2017theory,abduljabbar2019extreme}.
There are a variety of \bem implementations available; one that is most notable is \bem++, which includes high-order elements \cite{smigaj2015solving} with extensions for adaptivity added in \cite{bespalov2019adaptive,betcke2019adaptive}.
%
In this paper, we focus on the \emph{\nystrom} discretization, in which the integral in the equation is replaced by its quadrature approximation.  
The \nystrom method is simple, yet it enables very efficient methods to solve the discretized integral equation.
Compared to \bem methods, \nystrom methods tend to be more efficient, especially for changing or moving surfaces.
However, \nystrom methods are more difficult to apply to non-smooth surfaces (we do not consider high-order methods for surfaces with sharp edges and corners in this work).


The key element of \nystrom methods for \bie equations is efficient quadrature rules for singular and near-singular integrals.
In the \bie literature, such integration schemes fall into one of the several categories: \textit{singularity cancellation},  \textit{asymptotic correction}, \textit{singularity subtraction}, \textit{custom quadratures} or approximation-based quadrature schemes.

\emph{Singularity cancellation schemes} apply a change of variables to remove the singularity in the layer potential, allowing for the application of standard smooth quadrature rules. 
The first polar change of variables was detailed in the context of acoustic scattering \cite{bruno2001fast}, which leveraged a partition of unity and a polar quadrature rule to remove the singularity in the integrand of layer potential.
 the method was extended to open surfaces in \cite{bruno2013high}.
This methodology was applied to general elliptic \pdes in \cite{YBZ} and coupled with the kernel-independent fast multipole method \cite{ying2004kernel} and a general $C^\infty$ surface representation for complex geometries \cite{ying2004simple}. Its advantages and disadvantages compared to \qbkix are discussed in \cref{sec:results}.
Recently, \cite{malhotra2019taylor} demonstrated that the choice of partition of unity function used for the change of variables has a dramatic effect on overall convergence order.
The first singularity cancellation scheme in \threed on general surfaces composed of piecewise smooth triangles was presented in \cite{bremer2012nystrom,bremer2013numerical}. 
\cite{ganesh2004high} introduced a change of variables method for acoustic scattering on \threed surfaces, parametrized by spherical coordinates by integrating over a rotated coordinate system that cancels out the singularity.

\emph{Asymptotic correction methods} study the inaccuracies due to the singular \pde kernel with asymptotic analysis and apply a compensating correction. 
\cite{beale2004grid,beale2016simple,tlupova2019regularized} compute the integral with a regularized kernel and add corrections for regularization and discretization for the single and double layer Laplace kernel in \threed, along with the Stokeslet and stresslet in \threed.
\cite{carvalho2018asymptotic1} computes an asymptotic expansion of the kernel itself, which is used to remove the aliasing error incurred when applying smooth quadrature rules to near-singular layer potentials. 
This method is extended to \threed in \cite{carvalho2018asymptotic} and a complete asymptotic analysis of the double-layer integral is performed in \cite{khatri2020close}.
\emph{Singularity subtraction methods} \cite{jarvenpaa2003singularity, jarvenpaa2006singularity} explicitly subtract the singular component of the integrand analytically, which produces a smooth bounded integral that can be integrated with standard quadrature rules. 
However, the analytic calculations involved in these approaches are often tailored to a particular \pde and require recalculation for each new \pde of interest.

\emph{Custom quadrature rules} aim to integrate a particular family of functions to high-order accuracy. This can allow for arbitrarily accurate and extremely fast singular integration methods, since the quadrature rules can be precomputed and stored  \cite{alpert1999hybrid, xiao2010numerical}. 

Our method falls into the final category: \emph{approximation-based quadrature schemes}. 
The first use of a local expansion to approximate a layer potential near the boundary of a \twod boundary was presented in \cite{barnett2014evaluation}. 
By using a refined, or \textit{upsampled}, global quadrature rule to accurately compute coefficients of a Taylor series, the resulting expansion serves as a reasonable approximation to the solution near the boundary where quadrature rules for smooth functions are inaccurate.
This scheme was then adapted to evaluate the solution both near and on the boundary, called Quadrature by Expansion (\qbx) \cite{KBGN,EGK}. 
The first rigorous error analysis of the truncation error of \qbx was carried out in \cite{EGK}.

A fast implementation of \qbx in \twod, along with a set of geometric constraints required for well-behaved convergence, was presented in \cite{RKO}.
However, the interaction of the expansions of \qbx and the translation operator expansions of the \fmm resulted in a loss of accuracy, which required an artificially high multipole order to compensate for this additional error.
\cite{wala2018fast} addresses this shortcoming by enforcing a confinement criteria on the location of expansion disks relative to \fmm tree boxes.
\cite{aT2} provided extremely tight error heuristics for various kernels and quadrature rules in \twod using contour integration and the asymptotic approach of \cite{elliott2008clenshaw}.
\cite{aT1} then leveraged these estimates in a \qbx algorithm for Laplace and Helmholtz problems in \twod that adaptively selects quadrature upsampling and the expansion order for each \qbx expansion.
In the spirit of \cite{ying2004kernel}, \cite{RBZ} generalizes \qbx to any elliptic \pde by using potential theory to form a local, least-squares solution approximation using only evaluations of the \pde's fundamental solution.

The first extension of \qbx to \threed  was \cite{ST}, where the authors present a \textit{local, target-specific } \qbx method on spheroidal geometries.
In a local \qbx scheme, an upsampled accurate quadrature is used as a local correction to the expansion coefficients computed from the coarse quadrature rule over the boundary.
This is in contrast with a \textit{global} scheme, where the expansion coefficients are computed from the upsampled quadrature with no need for correction.
The first local \qbx scheme appears in \cite{barnett2014evaluation} in \twod, but the notion of local \fmm corrections dates back to earlier work such as \cite{alpert1999hybrid, kapur1997high}.
The expansions in \cite{ST} computed in a target-specific \qbx scheme can only be used to evaluate a single target point, but each expansion can be computed at a lower cost than a regular expansion valid in a disk.
The net effect of both these algorithmic variations are greatly improved constants, which are required for complicated geometries in \threed.
\cite{wala20193d} extends the \qbx-\fmm coupling detailed in \cite{wala2018fast} to \threed surfaces, along with the geometric criteria and algorithms of \cite{RKO} that guarantees accurate quadrature. 
\cite{wala2019optimization} improves upon this by adding target-specific expansions to \cite{wala20193d}, achieving a 40\% speed-up and \cite{wala2020approximation} provides a thorough error analysis of the interaction between computing \qbx expansions and \fmm local expansions.

In addition to techniques described above, a singular quadrature scheme of \cite{HO}, further extended to \twod Stokes flows in \cite{wu2020solution} and to near-singular \threed line integrals in \cite{klinteberg2019accurate},  does not fit into one of the above categories. 
While this method performs exceptionally well in practice, it does not immediately generalize to \threed surfaces in an efficient manner.

Most techniques mentioned above assume smooth domain boundaries or use adative refinement to handle non-smooth features.
There has been a great deal of recent work on special quadratures for regions with corners \cite{SR,S,serkh2018solution,hoskins2019numerical,rachh2017solution,serkh2016solution}.
Although not yet generalized to \threed, this work has the potential to vastly improve the performance of \threed \nystrom boundary integral methods on regions with corners and edges.

A way to avoid singular quadratures entirely is to use the \textit{method of fundamental solutions} (\mfs), which represents the solution as a sum of point charges on an equivalent surface outside of the \pde domain. 
\mfs was successfully applied in \twod \cite{barnett2008stability} and in axis-symmetric \threed problems \cite{liu2016efficient}.
Recently, \cite{gopal2019solving} has introduced an \twod approach similar in spirit to \mfs, but reformulated as a rational approximation problem.
Eliminating the need for singular integration makes these methods advantageous, but placing the point charges robustly can be challenging in practice and general \threed geometries remain a challenge.

We also briefly mention the use of \emph{isogemetric analysis} (\iga)\cite{hughes2005isogeometric} in the context of boundary integral equations.
\iga aims to use  the same basis functions for geometry and solution representation, in particular, similar to our work, reducing the gap between representations
used in CAD, and those needed for high-order \bem.  \iga has  been successfully applied to singular and hypersingular boundary integral equations with a collocation discretization \cite{taus2016isogeometric}.  A \nystrom \iga method coupled with a regularized quadrature scheme is detailed in \cite{zechner2016isogeometric}.

The rest of the paper is organized as follows: 
In \cref{sec:formulation}, we briefly summarize the problem formulation, geometry representation and discretization.
In \cref{sec:algo}, we detail our singular evaluation scheme and with algorithms to enforce admissibility, adaptively upsample the boundary discretization, and query surface geometry to evaluate singular/near-singular integrals.
In \cref{sec:error}, we provide error estimates for \qbkix. 
In \cref{sec:complexity}, we summarize the complexity of each of the algorithms described in \cref{sec:algo}.
In \cref{sec:results}, we detail convergence tests of our singular evaluation scheme and compare against other state-of-the-art methods.

%% file: formulation.tex
\section{Formulation\label{sec:formulation}}
\subsection{Problem Setup}

We restrict our focus to interior
Dirichlet boundary value problems of the form
\begin{align}
  L u(\vx) =  0,& \quad \vx \in \Omega,\\
  u(\vx) =  f(\vx),& \quad \vx \in \partial\Omega = \Gamma,
  \label{eq:pde}
\end{align}
with multiply- or singly-connected domain $\Omega$ of arbitrary genus. 
Our approach applies directly to standard integral equation formulations of exterior Dirichlet and Neumann problems; we include results for an exterior Dirichlet problem in \cref{sec:results-torii}.
Here $L$ is a linear elliptic operator and $f$ is at least $C^k$.
While our method can be applied to any non-oscillatory elliptic \pde, we use the following equations in our examples: 
\begin{equation}
  Lu =\begin{cases}
    \Delta u & \text{Laplace}\\
    \Delta u  - \nabla p, \quad \nabla \cdot u = 0 & \text{Stokes}\\
    \Delta u  + \frac{1}{1-2\nu}\nabla \nabla\cdot u & \text{Navier (linear elasticity)}\\
  \end{cases}
  \label{eq:pdes}
\end{equation}

We follow the approach of \cite{YBZ}. 
We can express  the solution at a point $\vx \in \Omega$ in terms of the double-layer potential
\begin{equation}
  u(\vx) = D[\phi](\vx) = \int_\Gamma \frac{\partial G(\vx,\vy)}{\partial \vn(\vy)} \phi(\vy) d\vy_\Gamma,
  \label{eq:double_layer}
\end{equation}
where $G(\vx,\vy)$ is the \textit{fundamental solution} or \textit{kernel} of \cref{eq:pde}, $\vn(\vy)$ is the normal at $\vy$ on $\Gamma$ pointing into the exterior of $\Omega$, and $\phi$ is an unknown function, or \textit{density}, defined on $\Gamma$.
We list the kernels associated with the \abbrev{PDE}s\xspace in \cref{eq:pdes} in \cite[Section 1]{morse2020bsupplementary}.
Using the jump relations for the interior and exterior limits of $u(\vx)$ as $\vx$ tends towards $\Gamma$ \cite{K,mikhlin2014integral,pozrikidis1992boundary,parton1982integral}, we know that \cref{eq:double_layer} is a solution to \cref{eq:pde} if $\phi$ satisfies 
\begin{equation}
  \left(\frac{1}{2}I + D + M\right)[\phi](\vx) = f(\vx), \vx \in \Gamma
  \label{eq:int-eq}
\end{equation}
with identity operator $I$. 
We will refer to $\phi$ as the \textit{density} and $u(\vx)$ as the \textit{potential} at $\vx$.
The double-layer integrals in this equation are \textit{singular}, due to the singularity in the integrand of \cref{eq:double_layer}. 
Additionally, as $\vx$ approaches $\Gamma$, \cref{eq:double_layer} becomes a \textit{nearly singular} integral.

The operator $M$ completes the rank of $\frac{1}{2}I + D$ to ensure invertibility of \cref{eq:int-eq}. 
If $\frac{1}{2}I + D$ is full-rank, $M = 0$.
When $\frac{1}{2}I + D$ has a non-trivial null space, $M$ accounts for the additional constraints to complete the rank of the left-hand side of \cref{eq:int-eq}.
For example, for the exterior Laplace problem on $\ell$ multiply-connected domains, the null space of $\frac{1}{2}I + D$ has dimension $\ell$ \cite{ST}.
The full set of cases for each kernel is considered in this work and their corresponding values of $M$ have been detailed in \cite{YBZ}. 


\subsection{Geometry representation \label{sec:geom-def}}

\begin{figure}[!htb]
  \centering
  \begin{minipage}{\textwidth}
      \includegraphics[width=\linewidth]{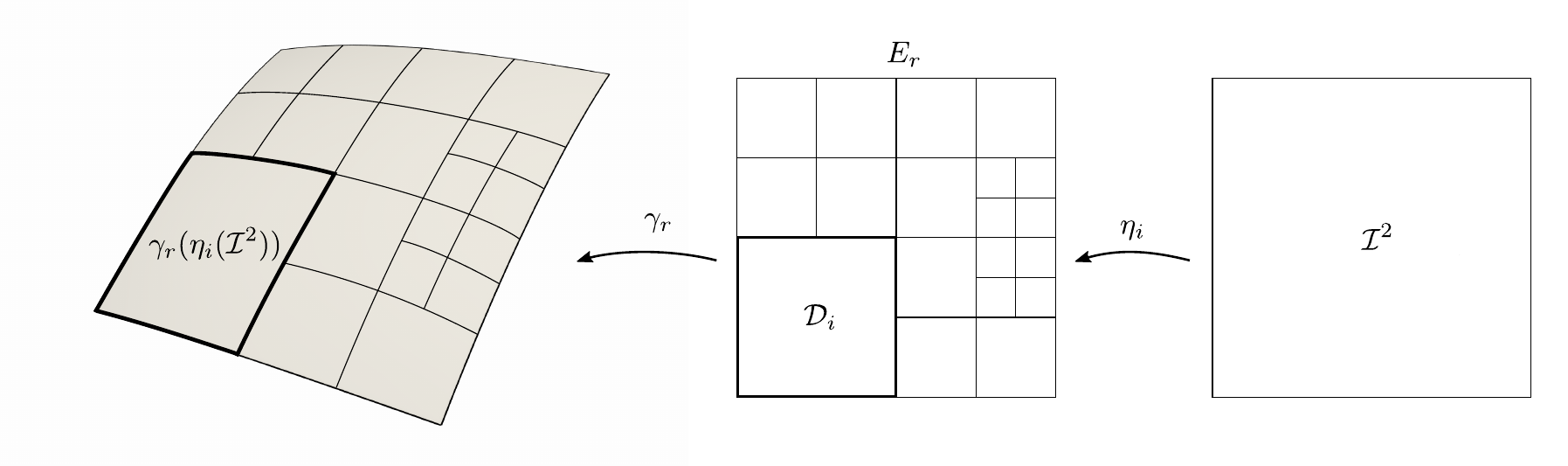}
  \end{minipage}\hfill
  \mcaption{fig:geometry-representation}{Patch Quadrisection}{
      Right: the standard domain $\mathcal{I}^2$ of a single surface or quadrature patch.
    Middle: a collection of subdomains $\mathcal{D}_i$ of $E_r$, produced by quadrisection. 
    Each $\mathcal{D}_i$ corresponds to a map $\eta_i$ such that $\mathcal{D}_i = \eta_i(\mathcal{I}^2)$; a single $\mathcal{D}_i$ is highlighted in bold.
    Left: the image of $E_r$ under the patch $\gamma_r$. 
    The final image of each subdomain is outlined, with the image of $\mathcal{D}_i$ in bold.
  }
\end{figure}
We assume that the smooth domain boundary $\Gamma$ is given by a \textit{quadrilateral mesh} consisting of quadrilateral faces $Q_r$, referred to as \textit{quads}.
Each quad is associated with a parametric domain $\I^2 =  [-1,1]^2 = E_r$, along with embeddings $\gamma_r : E_r \to \mathbb{R}^3$ for each quad such that $Q_r = \gamma_r(E_r)$.  
We assume that the quad mesh is \textit{conforming}, i.e., two non-disjoint faces either share a whole edge or a single vertex; examples of this are shown in \Cref{fig:greens-id-test-cases,fig:solver-conv-test-cases}.
We assume that no two images $\gamma_r(E_r)$ intersect, except along the shared edge or vertex.
The surface $\Gamma$ is the union of patches $\cup_r \gamma_r(E_r) = \cup_r Q_r$.
We also assume that $\Gamma$ is sufficiently smooth to recover the solution of \cref{eq:pde} up to the boundary \cite{K} and is at least $C^k$. 

To represent the surface geometry, we approximate $\Gamma$ with a collection of \emph{B\'ezier patches}, given by a linear combination of tensor-product Bernstein polynomials
\begin{equation}
    \vP_i(s,t) = \sum_{\ell =0}^n\sum_{m =0}^n \vector{a}^{(i)}_{\ell m }B_\ell^n(s)B_m^n(t),
  \label{eq:tensor-product}
\end{equation}
where  $B_\ell^n(t) = \binom{n}{\ell} t^{n-\ell}(1-t)^{\ell}$ for each $\ell$ are the $n$-th degree Bernstein polynomials, $i$ denotes the index of a patch in the collection and $\vector{a}_{\ell m}^{(i)} \in \mathbb{R}^3$.
Each patch $\vP$ is a vector function from $\I^2$ to $\mathbb{R}^3$, so $s,t\in [-1,1]$.
We will refer to this approximation of $\Gamma$ as $\hat{\Gamma}$. 

The domain $E_r$ of each embedding function $\gamma_r$ is adaptively refined using \emph{quadrisection}, i.e., splitting a square domain into four square subdomains of equal size. 
Quadrisection induces a \textit{quadtree} structure on each $E_r$. 
The root of the quadtree is the original domain $\I^2$ and each node of the tree is related by a single quadrisection of a subdomain of $E_r$. 
The leaves of the quadtree form a collection of subdomains $\mathcal{D}_i$ whose union equals $E_r$, as shown in \cref{fig:geometry-representation}-middle.
Given an indexing scheme of all $\mathcal{D}_i$'s over all $E_r$'s, we define the function $r(i)$ that maps the leaf node index $i$ to its root node index $r$ in the quadtree forest, indicating that $\mathcal{D}_i \subset E_r$. 
For each $r$, $E_r$ can have a distinct sequence of associated quadrisections and therefore a distinct quadtree structure.
We refer to the process of \textit{refinement} or \textit{refining a patch $\vP$} as the construction of such quadtrees for each $E_r$ subject to some set of criteria.

On each $\mathcal{D}_i$ at the quadtree leaves, we define a B\'ezier patch and reparametrize each patch over $\I^2$ by defining the affine map $\eta_i: \I^2 \to E_{r(i)}$ such that $\eta_i(\I^2) = \mathcal{D}_i \subseteq E_{r(i)}$.
It follows that the set of subdomains $\{ \eta_i(\I^2) \,| \,r(i) = \kappa\}$ form a cover of $E_\kappa$ and $\{ \gamma_\kappa(\eta_i(\I^2))\, | \,r(i) = \kappa\}$ likewise covers $\gamma_\kappa(E_\kappa)$.
We summarize this setup in \Cref{fig:geometry-representation}; examples of surfaces of this form can be seen in \Cref{fig:greens-id-test-cases,fig:solver-conv-test-cases,fig:torii,fig:vessel}.

\subsection{Problem discretization \label{sec:discretization}}
We use two collections of patches in the form described above: $\Pcoarse$ and $\Pfine$.
The patches in $\Pcoarse$, called \emph{surface patches}, determine $\Gammah$ from $\Gamma$ and the set of patches $\Pfine$, called \emph{quadrature patches}, are obtained by further quadrisection of the surface patches in $\Pcoarse$.
The geometry of $\Gammah$ is not changed by this additional refinement of $\Pcoarse$, but the total number of subdomains $E_{r(i)}$ is increased.
We will detail the geometric criteria that $\Pcoarse$ and $\Pfine$ must satisfy in \cref{sec:geom_criteria}.
Discretizing $\Gammah$ with with a quadrature rule based on $\Pfine$ results in a denser sampling of $\Gammah$ than a similar discretization of $\Pcoarse$.
We will refer to $\Pcoarse$ as the \textit{coarse discretization} of $\hat{\Gamma}$ and $\Pfine$ as the \textit{upsampled} or \textit{fine discretization} of $\hat{\Gamma}$.

We index the patches in $\vP_i \in \Pcoarse$ by $i = 1,\hdots N$; we can then rewrite \cref{eq:double_layer} as a sum of integrals over surface patches:
\begin{equation}
  u(\vx) = \sum_{i=1}^N\int_{\vP_i} \frac{\partial G(\vx,\vy)}{\partial \vn(\vy)} \phi(\vy) d\vy_{\vP_i}.
  \label{eq:double_layer_patches} 
\end{equation}






We discretize functions defined on $\Gammah$, such as
\cref{eq:double_layer_patches}, at $q$-node composite tensor-product
Clenshaw-Curtis quadrature points on $\I^2$ of patches in $\Pcoarse$.  
We refer to these points and weights on a single patch $\vP_i$ as $x_j$ and $w_j^{\lbl{CC}}$ respectively, for $j = 1\ldots q^2$. 
The quadrature point $\vy_{ij}$ from $\vP_i$ is defined as $\vy_{ij} = \vP_{i}(\eta_i(x_j))$. 
We assume that the boundary condition $f$ is given by a black-box evaluator on $\mathbb{R}^3$ that can be used to obtain values at $\vy_{ij}$.
For clarity, we reindex the surface points by a global index $I = 1, \hdots, q^2N$.
We discretize the double layer integral \cref{eq:double_layer_patches} on $\Pcoarse$ to approximate the solution $u(\vx)$:
\begin{equation}
    u(\vx,\Pcoarse) \approx \hat{u}(\vx,\Pcoarse) = \sum_{i=1}^N\sum_{j=1}^{q^2} \frac{\partial G(\vx,\vy_{ij})}{\partial \vn(\vy_{ij})} \phi_{ij} \sqrt{g_{ij}}w_j^{\lbl{CC}} 
    = \sum_{I=1}^{q^2N}\frac{\partial G(\vx,\vy_I)}{\partial \vn(\vy_I)} \phi_I \hat{w}_I
  \label{eq:double_layer_disc}
\end{equation}
with $g_{ij}$ being the determinant of the metric tensor of $\vP_i$ at $x_j$ and $\hat{w}_{i\cdot q^2+j} = \sqrt{g_{ij}}w_{j}^{\lbl{CC}}$.
In other words, $\hat{u}(\vx, \Pcoarse) = \hat{D}[\phi](\vx)$, where $\hat{D}[\phi](\vx) \approx D[\phi](\vx)$.

We can also discretize functions with tensor-product Clenshaw-Curtis nodes on the domains of patches in $\Pfine$.
The values of functions on $\Pfine$ are \emph{interpolated} from their values on the quadrature nodes of $\Pcoarse$ rather than being computed directly on $\Pfine$.
We call this interpolation from $\Pcoarse$ to $\Pfine$ \textit{upsampling}.
We denote the quadrature nodes and weights on $\Pfine$ by $\tilde{x}_j$ and $\tilde{w}_j$ with a similar global index $J$ and refer to them as the \textit{upsampled} nodes and weights.
Identical formulas are used for computing quadrature on $\Pfine$ with the nodes and weights $\tilde{x}_j$, $\tilde{w}_j$ on $\Pfine$, denoted $u(\vx,\Pfine)$ and $\hat{u}(\vx,\Pfine)$, repsectively.

In the next section, we describe the algorithm to compute an accurate approximation to the singular/near-singular double-layer integral in \cref{eq:double_layer}, using a quadrature rule for smooth functions (\cref{eq:double_layer_disc}) as a building block. 
This algorithm allows us to compute the matrix-vector products $A\phi$, for a vector of values $\phi$ defined at the quadrature points $\vy_I$, where $A$ is the discrete operator obtained from the left-hand side of \cref{eq:int-eq} after approximating $D[\phi](\vy)$ with the singular integration scheme.
As a result, we can solve the linear system using \gmres, which only requires a matrix-vector product 
\begin{equation}
  A\phi = f,
  \label{eq:linear_system}
\end{equation}
where $f$ is the boundary condition sampled at the points $\vy_I$. The evaluation of these integrals is accelerated in a standard manner using the fast multipole method (\fmm)\cite{MB,ying2004kernel,greengard1987fast}.

%% file: algo.tex
\section{Algorithms\label{sec:algo}}

We now detail a set of algorithms to solve the integral equation in \cref{eq:int-eq} and evaluate the solution via the double layer integral in \cref{eq:double_layer} at a given target point $\vx \in \Omega$.
As described in the previous section, both solving \cref{eq:int-eq} and evaluating \cref{eq:double_layer} require accurate evaluation of singular/near-singular integrals of functions defined on the surface $\Gammah$.
We first outline our unified singular/near-singular integration scheme, \qbkix, its relation to existing approximation-based quadrature methods and geometric problems that can impede accurate solution evaluation.
We then describe two geometry preprocessing algorithms, \textit{admissibility refinement} and \textit{adaptive upsampling}, that address these issues to obtain the sets of patches $\Pcoarse$ and $\Pfine$ used by \qbkix.

\subsection{Singular and Near-Singular Evaluation \label{sec:singular-eval}}

We begin with an outline of the algorithm.
For a point $\vsx \in \Gammah$ on a patch $\vP$ from $\Pcoarse$ that is closest to $\vx$, we first upsample the density $\phi$ from $\Pcoarse$ to $\Pfine$ and compute the solution at a set of points $\vc_s$, $s = 1, \hdots p$ called \emph{check points}, sampled along the surface normal at $\vsx$ away from $\Gammah$. 
We use \cref{eq:double_layer_disc} to approximate the solution at the check points.
We then extrapolate the solution to $\vx$.

For a given surface or quadrature patch $\vP: \I^2 \rightarrow \mathbb{R}^3$,  we define the \textit{characteristic length} $L(P)$ as the square root of the surface area of $\vP$, i.e., $L(\vP) = \sqrt{\int_{\vP}d\vy_{\vP}}$.
We use  $L = L(\vP)$ or $L_\vy$ for $\vy \in \vP(D)$ to denote the characteristic length when $\vP$ is clear from context.
For a point $\vx \in \Omega$, we assume that there is a single closest point $\vsx \in \Gammah$ to $\vx$; all points to which the algorithm is applied will have this property by construction.  
Note that $\vn(\vsx)$, the vector normal to $\Gammah$ at $\vsx$, is chosen to point outside of $\Omega$.


We define three zones in $\Omega$ for which \cref{eq:double_layer} is evaluated differently in terms of \cref{eq:double_layer_disc} and the desired solution accuracy $\etrg$ .
The \emph{far field}  $\Omega_F = \{\vx \in \Omega \,|\, \|u(\vx) - \hat{u}(\vx;\Pcoarse)\|_2 \leq \etrg\}$, where the quadrature rule corresponding to $\Pcoarse$ is sufficiently accurate, and the \emph{intermediate field}  $\Omega_I= \{\vx \in \Omega \,|\, \|u(\vx) - \hat{u}(\vx;\Pfine)\|_2 \leq \etrg\}$, where quadrature over $\Pfine$ is sufficiently accurate.
The remainder of $\Omega$ is the \emph{near field} $\Omega_N = \Omega \setminus \Omega_I$.




\paragraph*{Non-singular integration}
To compute the solution at points $\vx$ in $\Omega_F$, \cref{eq:double_layer_disc} is accurate to $\etrg$, so we can simply compute $\hat{u}(\vx, \Pcoarse)$ directly.
Similarly for points in $\Omega_I \setminus \Omega_F$, we know by definition that $\hat{u}(\vx, \Pfine)$ is sufficiently accurate, so it can also be applied directly. 

\paragraph*{Singular/near-singular integration algorithm}

\begin{figure}[!htb]
  \begin{minipage}{\textwidth}
  \centering
      \includegraphics[width=.7\linewidth]{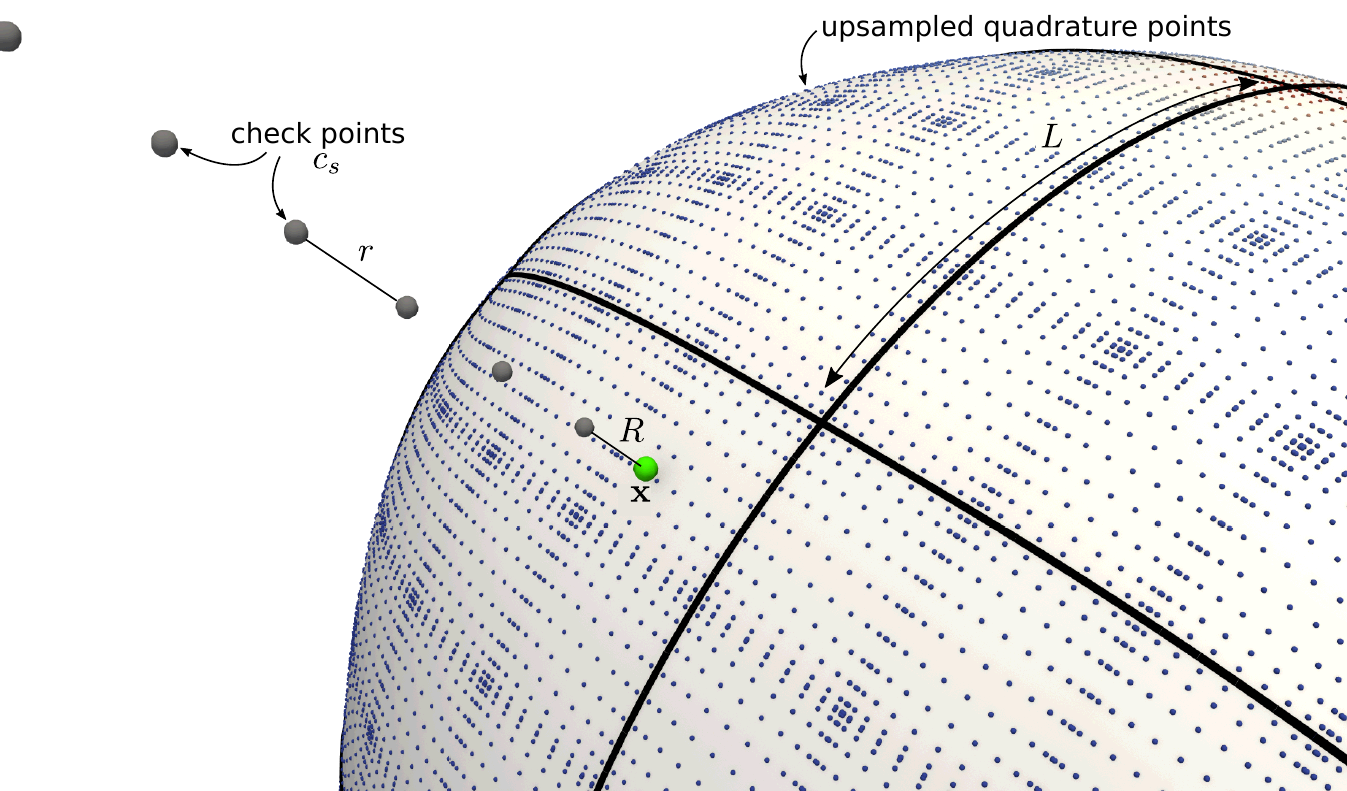}
  \end{minipage}\hfill
  \mcaption{fig:qbkix-schematic}{Schematic of singular/near-singular
  evaluation}{A small piece of a boundary $\Gammah$ is shown, along with the set of
  patches $\Pcoarse$ (patch boundaries are drawn in black). The target point $\vx$, in this
  case on $\Gammah$, is shown in green. The solution is evaluated
  at the check points $\vc_s$ (gray points off-surface) using the fine
  discretization $\Pfine$ (small dots on-surface). The distance from the first
  check point $\vc_0$ to $\Gammah$ is $R$ and the distance between consecutive
  check points $\vc_i$ and $\vc_{i+1}$ is $r$. In this example, $\Pfine$ is computed
  from $\Pcoarse$ with two levels of uniform quadrisection, producing 16 times
  more patches. The patch length $L$ is roughly proportional to the average edge
  length of the patch.}
\end{figure}

For the remaining points in $\Omega_N$, we need an alternative means of evaluating the solution.
In the spirit of the near-singular evaluation method of \cite{YBZ},  we construct a set of \textit{check points} $\vc_0, \hdots, \vc_p$ in $\Omega_I$ along a line intersecting $\vx$ to approximate the solution near $\vx$. 
However, instead of interpolating the solution as in \cite{YBZ}, we instead extrapolate the solution from the check points to $\vx$.
We define two distances relative to $\vsx$: $R(\vsx) =b L_\vsx = \|\vc_0 - \vsx\|_2$, the distance from the first check point $\vc_0$ to $\Gammah$,  and $r(\vsx) =a L_\vsx = \|\vc_i - \vc_{i+1}\|_2$, the distance between consecutive check points.
We assume  $0<a,b <1$.

The overall algorithm for the unified singular/near-singular evaluation scheme is as follows.
A schematic for \qbkix is depicted in \Cref{fig:qbkix-schematic}.

\begin{enumerate}
  \item Find the closest point $\vsx$ on $\Gammah$ to $\vx$.
  \item Given values $a$ and $b$, generate check points $C = \{\vc_0, \hdots, \vc_{p}\}$ 
    \begin{equation}
      \vc_s = \vsx -(R(\vsx) +s r(\vsx)) \vn(\vsx), \quad s=0, \hdots, p
      \label{eq:check}
    \end{equation}
    The center of mass of these check points $\vhc$ is called the \textit{check center} for $\vx$.
    Note that $\Pfine$ must satisfy the condition that $\vc_s$ are in $\Omega_I$ for a given choice of $a$ and $b$.
\item Upsample $\phi$. 
  We interpolate the density values $\phi_I$ at $x_I$ on patches in $\Pcoarse$ to quadrature points $\tilde{x}_J$ on patches in $\Pfine$ 
  with global indices $I$ and $J$ on $\Pcoarse$ and $\Pfine$ respectively.
  If a patch $\vP_i$ in $\Pcoarse$ is split into $m_i$ patches in $\Pfine$, we are interpolating from $q^2$ points to $m_iq^2$ points.
  \item Evaluate the potential at check points via smooth quadrature with the upsampled density, i.e. evaluate $\hat{u}(\vc_s) = \hat{u}(\vc_s, \Pfine)$ for $s=0,\hdots, p$.
  \item Compute a Lagrange interpolant $\tilde{u}$ through the check points $\vc_0,\hdots, \vc_p$ and values $\hat{u}(\vc_0), \hdots, \hat{u}(\vc_{p})$ and evaluate at the interpolant at $\vx$:
      \begin{equation}
          \tilde{u}(\vx) = \sum_{s=0}^p \hat{u}(\vc_s)\ell_s(t_\vx),
      \end{equation}
        where $\ell_s(\vx)$ is the $s$th Lagrange basis function through the points $\vc_0,\hdots, \vc_p$, and $t_\vx\in \mathbb{R}$ is such that $\vx = \vsx - t_\vx\vn(\vsx)$ (see \cref{fig:extrap-err-setup} for a schematic of the check points).
    Since $\vx$ lies between $\vc_0$ and $\Gammah$, we are extrapolating when computing $\tilde{u}(\vx)$.
\end{enumerate}

\paragraph*{Ill-conditioning of the discrete integral operator}
This evaluation scheme can be used directly to extrapolate all the way to the surface and obtain the
values of the singular integral in \cref{eq:int-eq}.
However, in practice, due to a distorted eigenspectrum of this approximate operator, \gmres tends to stagnate at a level of error corresponding to the accuracy of \qbkix when it is used to compute the matrix-vector product.
This is a well-known phenomenon of approximation-based singular quadrature schemes; \cite[Section 3.5]{KBGN}\cite[Section 4.2]{RBZ} present a more detailed study.
To address this, we average the interior and exterior limits of the solution at the quadrature nodes, computed via \qbkix, to compute the on-surface potential and add $\frac{1}{2}I$ to produce the interior limit.
This shifts the clustering of eigenvalues from around zero to around $\frac{1}{2}$, which is ideal from the perspective of \gmres.
We call this \textit{two-sided} \qbkix, while the standard version described above is called \textit{one-sided} \qbkix.
We observe stable and consistent convergence of \gmres when two-sided \qbkix is used to evaluate the matrix-vector multiply to solve \cref{eq:linear_system}. 
In light of this, we always use two-sided \qbkix within \gmres and set the stopping tolerance for \gmres to $\err{\gmres}=10^{-12}$, regardless of the geometry, boundary condition or  quadrature order. 

\subsection{Geometric criteria for accurate quadrature\label{sec:geom_criteria}}
The accuracy of the method outlined above is controlled by two competing error terms: \textit{quadrature error} incurred from approximating the layer potential \cref{eq:double_layer} with \cref{eq:double_layer_disc} in Step 4 and \textit{extrapolation error} due to approximating the singular integral with an extratpolated value in Step 5.
Both errors are determined by the location of check points relative to the patches in $\Pcoarse$ and $\Pfine$ (see \Cref{heuristic:error_quad_high_order,thm:extrap_error}). 

\begin{figure}[!htb]
  \centering
  \setlength\figureheight{1.9in}
  \setlength\figurewidth{2.1in}
  \begin{minipage}{.33\textwidth}
      \includegraphics[width=\linewidth]{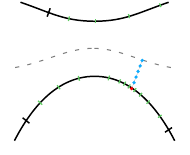}
  \end{minipage}\hfill
  \begin{minipage}{.33\textwidth}
    \includegraphics[width=\linewidth]{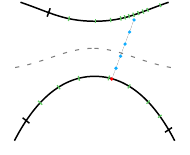}
  \end{minipage}\hfill
  \begin{minipage}{.33\textwidth}
    \includegraphics[width=\linewidth]{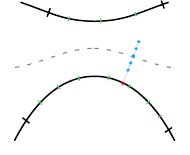}
  \end{minipage}\hfill
  \mcaption{fig:admissibility-motivation}{Possible check point configurations}{A \twod  example depicting three choices of $a$ and $b$ in \cref{eq:check}. 
 Shown is the boundary $\Gammah$, with black tick marks denoting patch boundaries of $\Pcoarse$, green tick marks denoting patch boundaries of $\Pfine$, the target point (red dots), its check points (blue dots) along the normal closest to the target point, and the medial axis of $\Gammah$ (gray dotted line).
Large (left) and small (middle) values of $a$ and $b$ can cause clustering of check points near to $\Gammah$, which requires large amounts of upsampling to compute the potential accurately. Using the medial axis as a heuristic to for admissibility (right), we can minimize the amount of adaptive upsampling required.}
\end{figure}
In \Cref{fig:admissibility-motivation}, we show three examples of different choices of check point locations to evaluate the potential at a point with \qbkix. 
In \cref{fig:admissibility-motivation}-left, $\vc_0$ is placed close to the target point, while in \cref{fig:admissibility-motivation}-middle, $\vc_0$ is far from the target point, but $\vc_p$ is close to a non-local piece of $\Gammah$. 
Both cases will require excessive refinement of $\Pcoarse$ in order to resolve \cref{eq:double_layer_disc} accurately with $\Pfine$.
On the other hand, in \cref{fig:admissibility-motivation}-right, we can either perform one refinement step on $\Pcoarse$ or adjust $a$ and $b$, which will result in fewer patches in $\Pfine$, and therefore provide a faster integral evaluation, while maintaining accuracy.

In an attempt to strike this balance between speed and accuracy, we need certain constraints on the geometry of $\Gammah$ to ensure the efficient and accurate application of \qbkix, which we impose on the patch sets $\Pcoarse$ and $\Pfine$.
We will first outline our constraints on the quadrature patch sets $\Pcoarse$ and $\Pfine$ which allow for accurate evaluation with \qbkix.

\subsubsection{Admissibility criteria\label{sec:admissible}}
A set of patches $\qP$ is \textit{admissibile} if the following statements are satisfied on each quadrature patch in $\qP$:
\begin{criteria}
  \item The error of a surface patch $\vP_i$ approximating an embedding $\gamma_r$ is below some absolute target accuracy $\err{g}$ \label{criteria:1}
  \item The interpolation error of the boundary condition $f$ is below some absolute target accuracy $\err{f}$ \label{criteria:2}
  \item For each check center $\vhc_j$ corresponding to the quadrature point $\vy_j$ on the surface, the closest point on $\hat{\Gamma}$ to $\vhc_j$ is $\vy_j$. \label{criteria:3}
\end{criteria}

\Cref{criteria:1} is required to ensure that $\Gammah$ approximates $\Gamma$ with sufficient accuracy to solve the integral equation.
We discuss how to choose $\err{g}$ in \cite[Section 6]{morse2020bsupplementary}; for the tests in this paper, we simply choose $\err{g} < \err{target}$.
\Cref{criteria:2} guarantees that $f$ can be represented at least as accurately as the desired solution accuracy.
We therefore similarly choose $\err{f} < \etrg$. 
\Cref{criteria:3}  balances the competing geometric constraints of cost and accuracy by flexibly placing check points as far as possible from $\Gammah$ without causing too much upsampling on other patches.
If a check point $\vc$ constructed from a surface patch $\vP$ is too close to another surface patch $\vP'$, \Cref{criteria:3} will indicate that $\vP$ is inadmissible. 
If $\vP$ is subdivided into its children, new check points $\vc^\prime$ generated from these children of $\vP$ will be closer to $\vP$ and further from $\vP'$.
Since check points are placed at distances proportional to $L(\vP)$, repeated refinement of $\vP$ will eventually satisfy \Cref{criteria:3}. 

\subsubsection{Upsampling criteria\label{sec:adaptive_upsampling}}
Once we have a set of admissible surface patches satisfying \Cref{criteria:1,criteria:2,criteria:3}, we need to determine the upsampled quadrature patches $\Pfine$ that ensure that the check points generated from $\Pcoarse$ are in $\Omega_I$, i.e., $\|u(\vc) - \hat{u}(\vc, \Pfine)\| < \etrg$.
To achieve this, we need a criterion to determine which patches are ``too close'' to a given check point for the error to be below $\err{target}$.
We make the following assumption about the accuracy of our smooth quadrature rule: \textit{\cref{eq:double_layer_disc} is accurate to $\err{target}$ at points further than  $L(\vP)$ from $\vP$, for $\err{target} > 10^{-12}$}.
This is motivated by \cite{aT2,barnett2014evaluation}, which demonstrate the rapid convergence of the layer potential quadrature error with respect to $\|\vx - \vsx\|_2$.  
For sufficiently high quadrature orders, such as $q=20$, this assumption seems to hold in practice.
We say that a point $\vx$ is \textit{near} to $\vP$ if the distance from $\vx$ to $\vP$ is less than $L(\vP)$; otherwise, $\vx$ is \textit{far} from $\vP$.
We would like all check points required for the singular/near-singular evaluation of the discretization of \cref{eq:double_layer} using \qbkix to be far from all patches in $\Pfine$.
If this is satisfied, then we know that the Clenshaw-Curtis quadrature rule will be accurate to $10^{-12}$ at each check point.

\subsection{Refinement algorithm preliminaries}
Computing the distance from a check point to a given patch is a fundamental step in verifying the constraints on $\Pcoarse$ and $\Pfine$ from \cref{sec:admissible,sec:adaptive_upsampling}. 
Before detailing our refinement algorithms to enforce these criteria, we introduce several geometric algorithms and data structures that will be used to compute the closest point on piecewise polynomial surfaces.

\subsubsection{\aabb trees\label{sec:aabb_trees}} 
In order to implement our algorithms to enforce admissibility efficiently, we use a fast spatial data structure to find the patches that are close to a query point $\vx$.
In \cite{RKO, wala20193d}, the quadtree and octree within an \fmm is extended to support the geometric queries needed for a fast \qbx  algorithm.
In this work, we use an axis-aligned bounding box (\aabb) tree, which is a type of bounding volume hierarchy \cite{samet2006foundations}, implemented in \texttt{geogram} \cite{geogram}.
An \aabb is a tree with nodes corresponding to bounding boxes and leaves corresponding to bounding boxes containing single objects. 
A bounding box $B_0$ is a child of another box $B_1$ if $B_0 \subset B_1$; the root node is a bounding box of the entire domain of interest.
Operations supported by \aabb trees include: (i) finding all bounding boxes containing a query point, (ii) finding all bounding boxes that intersect another query box, (iii) finding the closest triangle to a query point (because triangles have trivial bounding boxes). 
By decoupling geometric queries from fast summation, the individual algorithms can be more thoroughly optimized, in exchange for the additional memory overhead of maintaining two distinct data structures.
The query algorithm presented in \cite{lu2019scalable} likely has better parallel scalability, but \aabb trees are faster for small to medium problem sizes on a single machine due to less redundant computation.

To define an \aabb tree for our patch-based surface $\Gammah$, we make use of the following fact: the control points of a B\'ezier surface ($\vector{a}_{\ell m}$'s from \cref{eq:tensor-product}) form a convex hull around the surface that they define \cite{F}.
As a result, we can compute a bounding box of a surface or quadrature patch $\vP$ directly from the B\'ezier coefficients simply by computing the maximum and minimum values of each component of the $\vector{a}_{\ell m}$'s, as shown in \cref{fig:patch-coeffs-bbox}-middle.
This bounding box can then be inserted into the \aabb tree as a proxy for a surface or quadrature patch.
\begin{figure}[!htb]
  \centering
  \setlength\figureheight{1.9in}
  \setlength\figurewidth{2.1in}
  \begin{minipage}{.33\textwidth}
      \includegraphics[width=\linewidth]{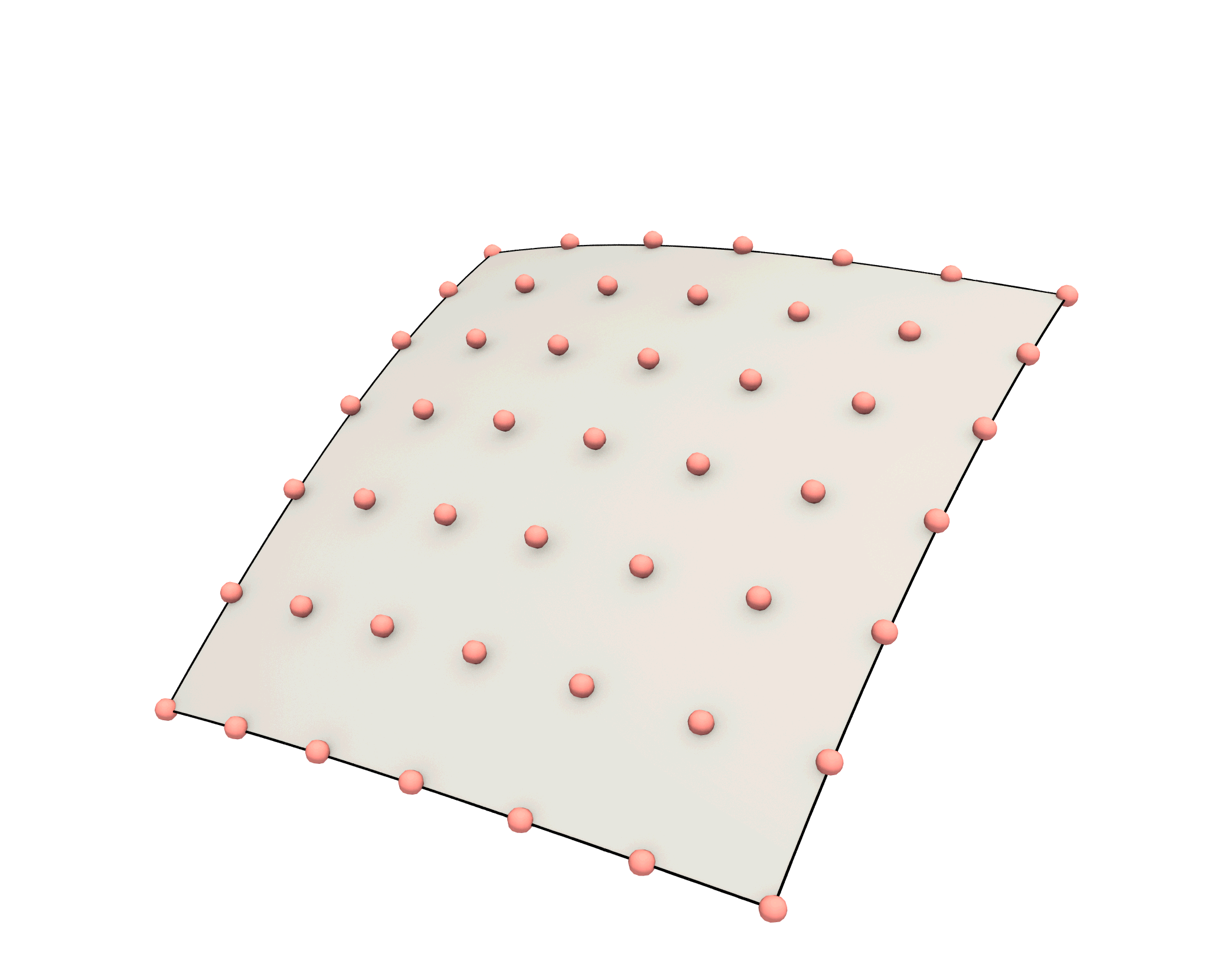}
  \end{minipage}\hfill
  \begin{minipage}{.33\textwidth}
    \includegraphics[width=\linewidth]{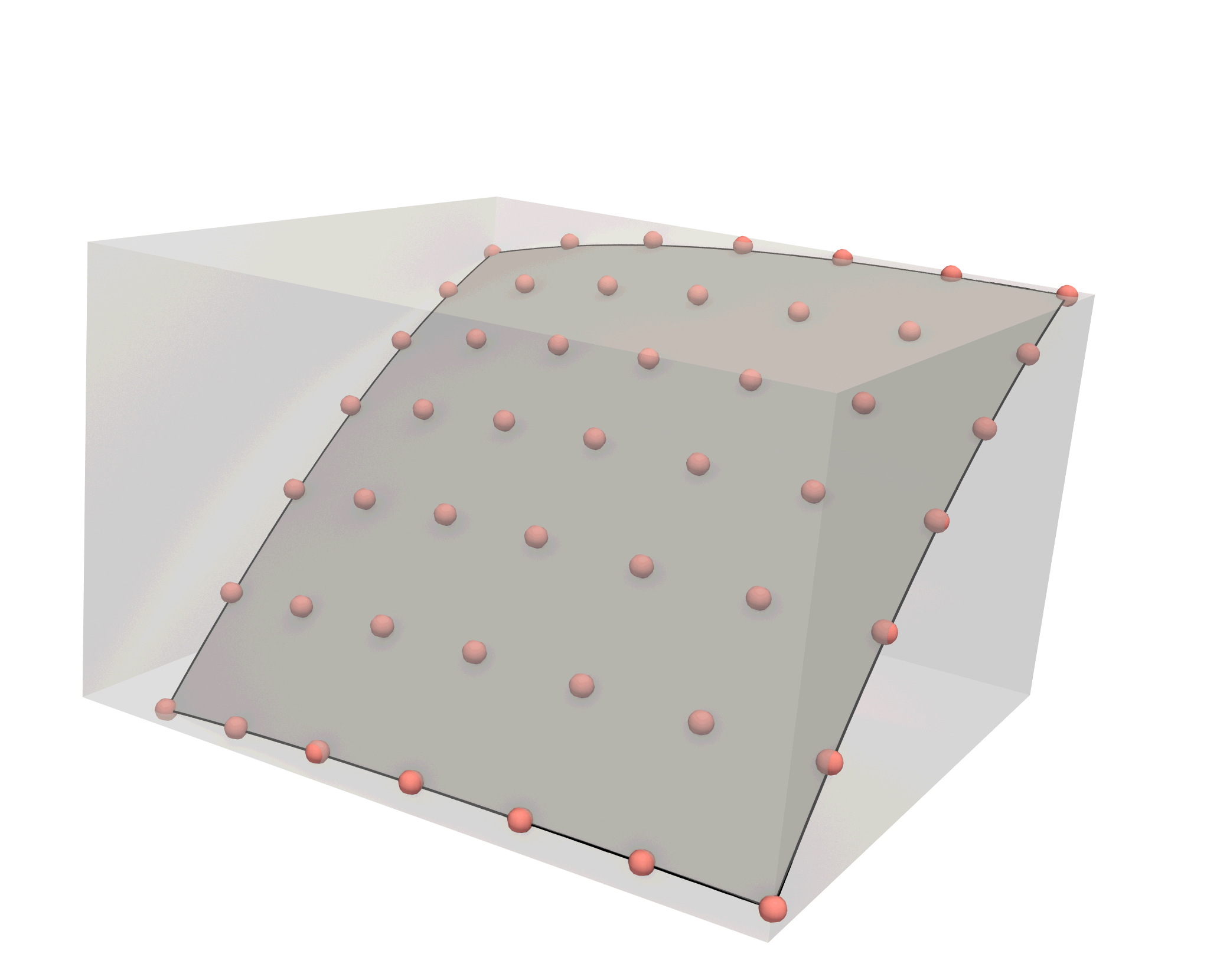}
  \end{minipage}\hfill
  \begin{minipage}{.33\textwidth}
    \includegraphics[width=\linewidth]{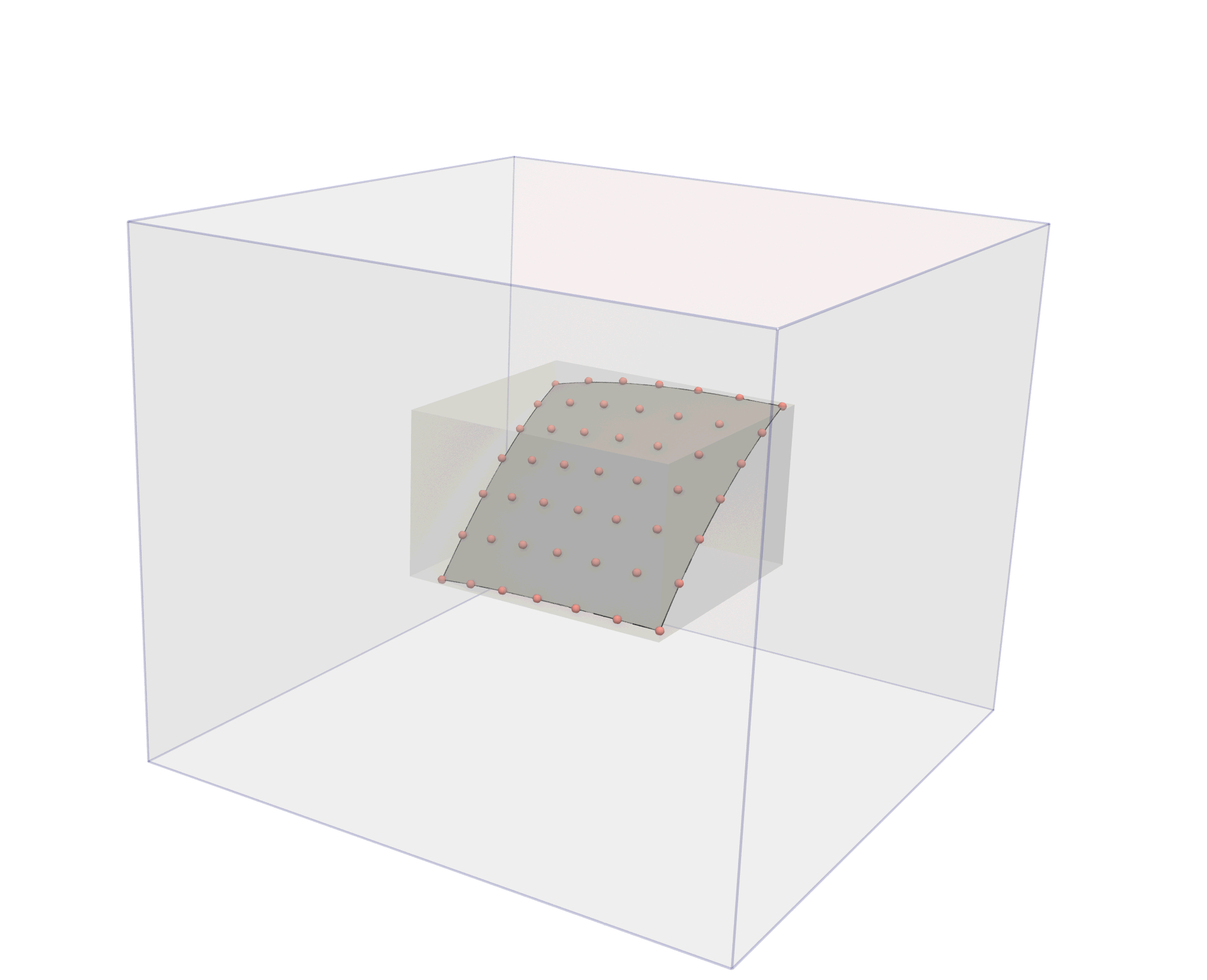}
  \end{minipage}\hfill
  \mcaption{fig:patch-coeffs-bbox}{Relationship between control points and
  bounding boxes}{Left: a patch in the tensor product B\'ezier basis, with control
      points ($\vector{a}_{\ell m}$'s from \cref{eq:tensor-product}) plotted. The convex hull of the control points of a patch are guaranteed
      to contain the patch. Center: The patch bounding box, computed from the control
      points. Right: The near-zone bounding box of the patch from
      \cref{sec:adaptive_upsampling_algo} computed by
      inflating the bounding box by $L(\vP)$.
}
\end{figure}

\subsubsection{Computing the closest point to a patch\label{sec:closest_point_algo}}

To find a candidate closest patch $\vP_{i_0}$ to $\vx$, we construct a fine triangle mesh and bounding boxes of each patch in $\Pcoarse$ and insert them into an \aabb tree.
We can query the \aabb tree for the nearest triangle to $\vx$ with the \aabb tree, which corresponds to $\vP_{i_0}$.
We then compute the accurate true distance $d_{i_0}$ to $\vP_{i_0}$ using a constrained Newton method, presented in detail in \cite[Section 2]{morse2020bsupplementary}.

However, there may be other patches whose distance to $\vx$ is less than $d_{i_0}$, as shown in \cref{fig:candidate-near-patch}.
To handle this case, we then query the \aabb tree for all patches $\vP_{i_1}, \hdots, \vP_{i_k}$ that are distance at most $d_{i_0}$ from $\vx$.
This is achieved by forming a query box centered at $\vx$ with edge length $2d_{i_0}$ and querying the \aabb tree for all intersection bounding boxes. 
The precise distance is then computed for each patch  $\vP_{i_1}, \hdots, \vP_{i_k}$ with \cite[Section 2]{morse2020bsupplementary} and the smallest distance is chosen.
We summarize this process in \cref{alg:closest_point}.

\begin{algorithm}[!htp]
    \KwData{A set of quadrature patches $\qP$, a query point $\vx$, Newton method tolerance $\err{opt}$}
  \KwResult{The closest point $\vsx$ on $\qP$ to $\vx$}

  \DontPrintSemicolon
  Construct an AABB tree $T_T$ from a fine triangle mesh of the quadrature patches of $\qP$\;
  Construct an AABB tree $T_B$ from bounding boxes of quadrature patches in $\qP$.\;
    $\tau_0 = $ closest triangle to $\vx$ computed with $T_T$ \;
  
    $\vP_{i_0} = $ patch corresponding to $\tau_0$\;
    Find the closest point $\vector{s}_{\vector{\vx},0}$ on $\vP_{i_0}$ to $\vx$ with \cite[Section 2]{morse2020bsupplementary}.\;
    $d_{i_0} = \|\vx - \vector{s}_{\vector{\vx},0}\|_2$\;
    $B_{d_{i_0}}(\vx)=$ a box centered a $\vx$ with edge length $2d_{i_0}$\;
    Find the boxes $B_{i_1}, \hdots B_{i_k}$ in $T_B$ that intersect $B_{d_{i_0}}(\vx)$\;
    
    \For{$B_{i_j} \in B_{i_1}, \hdots B_{i_k}$}{
      $\vP_{i_j} =$ quadrature patch corresponding to $B_{i_j}$ \;
      Find the closest point $\vector{s}_{\vector{\vx},j}$ on $\vP_{i_j}$ to $\vx$ with \cite[Section 2]{morse2020bsupplementary} to precision $\err{opt}$.\;
      $d_{i_j} = \|\vx - \vector{s}_{\vector{\vx},j}\|_2$\;
    }
    $j^* = \mathrm{argmin}_j\{d_{i_j}\}$ \;
    \Return{$\vector{s}_{\vector{x},j^*}$}
  \mcaption{alg:closest_point}{Compute the closest point to $\vx$}{}
\end{algorithm}

\begin{figure}[!htb]
  \centering
  \hfill
  \begin{minipage}{.2\textwidth}
      \includegraphics[width=\linewidth]{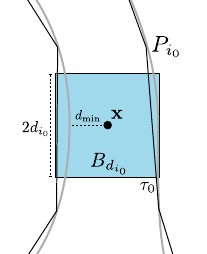}
  \end{minipage}\hfill
  \mcaption{fig:candidate-near-patch}{A \twod schematic of near-patch candidate selection}{
      A visual depiction of the quantities defined in lines 3-7 of \cref{alg:closest_point} (shown here in \twod for simplicity), with notation matching \cref{alg:point_marking}.
      The triangle-mesh proxy is drawn in as black lines and patches are drawn as gray curves.
      We have found an initial closest triangle $\tau_0$ to $\vx$ corresponding to patch $\vP_{i_0}$ and computed $d(\vx, \vP_{i_0}) = d_{i_0}$.
      We then query the \aabb tree for all patches that intersect box $B_{d_{i_0}}$ with edge length $2d_{i_0}$, shown in blue.
      There is clearly a patch that is closer to $\vx$ than $\vP_{i_0}$ that will be returned from the query, which will be distance $d_\lbl{min}$ from $\vx$.
}
\end{figure}

\subsection{Admissibility algorithm\label{sec:admissible_algo}}
Our algorithm to enforce \cref{criteria:1,criteria:2,criteria:3} proceeds as follows:
\begin{itemize}
    \item To enforce \Cref{criteria:1}, we adaptively fit a set of surface patches to the embeddings $\gamma_r$ representing $\Gamma$.
We construct a bidegree $(n,n)$ piecewise polynomial least-squares approximation $\vP_i$ in the form of \cref{eq:tensor-product} to $\gamma_r$ on $I^2$.
        If $\vP_i$'s domain $\mathcal{D}_i$ is obtained by refinement of $E_r$, we fit $\vP_i \circ \eta_i$ to $\gamma_r$ on $\I^2$, using $4n \times 4n$ samples on $\I^2$. 
        If the pointwise error of $\vP_i$ and its partial derivatives is greater than $\err{g}$, then it is quadrisected and the process is repeated. 

\item  Once the embeddings are resolved, we resolve $f$ on each surface patch produced from the previous step in a similar fashion to enforce \Cref{criteria:2}.
  However, rather than a least-squares approximation in this stage, we use piecewise polynomial interpolation.
 
\item To enforce \Cref{criteria:3}, we construct the set of check centers $\vhc_I$ which correspond to the check points required to evaluate the solution at the quadrature nodes $\vy_I$.
    For each check center $\vhc_I$, we find the closest point $\vector{s}_{\vhc_I} \in \Gammah$.
    If $ \|\vector{s}_{\vhc_I} - \vy_I\| \geq \err{opt}$, we split the quadrature patch $\vP$ containing $\vy_I$.
        The tolerance $\err{opt}$ is used in the Newton's method in \cite[Section 2]{morse2020bsupplementary}; we usually choose $\err{opt}=10^{-14}$.
Since $d(\vhc_I,\Gammah)$ is proportional to $L_{\vy_I}$, the new centers $\vhc_I$ for the refined patches  will be closer to the surface. 
    We use \cref{alg:closest_point} to compute $\vector{s}_{\vhc_I}$.
        However, in the case of check points, we can skip lines 1-6 to compute $d_{i_0}$,  since $\vhc_I$ is $R + r(p+1)/2$ away from $\vy_I \in \vP(D)$ by construction.
        We can apply lines 7-14 of \cref{alg:closest_point} with $d_{i_0} = R + r(p+1)/2$ to compute $\vector{s}_{\vhc_I}$.
\end{itemize}

We summarize the algorithm to enforce \Cref{criteria:3} in \cref{alg:admissibility}.
At each refinement iteration, the offending patches are decreased by quadrisection, which reduces the distance from the quadrature point $\vy_I$ to its checkpoints.
This eventually satisfies \Cref{criteria:3} and the algorithm terminates. 

\begin{algorithm}[!ht]
  \KwData{A set of quadrature patches $\qP$, optimization tolerance $\err{opt}$}
  \KwResult{An admissible set of quadrature patches $\qP$}

  \DontPrintSemicolon
  $\qP = \Pcoarse$\;
  Mark all patches in $\qP$ as inadmissible.\;
  
  \While{any patch in $\qP$ is inadmissible}{
    Construct an AABB tree $T$ as described in \cref{sec:closest_point_algo} from $\qP$\;
    \For{$\vP \in \qP$}{
        \If{$\vP$ is inadmissible}{
            Construct a set of check centers $C_\vP$ for each $\vy_J \in \vP(D)$ \;

            \For{$\vhc \in C_\vP$}{
                $d_{i_0} = R + r(p+1)/2$\;
                Compute $\vector{s}_\vhc$ with lines 7-14 of \cref{alg:closest_point} with precision $\err{opt}$ and $d_{i_0}$.\;
                \eIf{$\|\vector{s}_\vhc -\vy_J\|_2 < \err{opt}$}{
                    Mark $\vP$ as admissible.\;
                } {
                    Mark $\vP$ as inadmissible.\;
                    break   \tcp{only need one bad check center to mark $\vP$ for refinement}
                }
            }
        }
    }
    \For{$\vP \in \qP$}{
      \If{$\vP$ is inadmissible}{
        Split $\vP$ into its four child patches, mark each as inadmissible, and replace $\vP$ with its children in $\qP$.
      }
    }
  }
  \Return{$\qP$}
  \mcaption{alg:admissibility}{Enforce admissibility \Cref{criteria:3} on a set of quadrature patches}{}
\end{algorithm}

\subsection{Adaptive upsampling algorithm \label{sec:adaptive_upsampling_algo}}
Before detailing our upsampling algorithm to satisfy the criteria outlined in \cref{sec:adaptive_upsampling}, we must define the notion of a \textit{near-zone bounding box} of a quadrature patch $\vP$, denoted $B_\lbl{near}(\vP)$.
The near-zone bounding box of $\vP$ is computed as described in \cref{sec:aabb_trees}, but then is inflated by $2L(\vP)$, as shown in \cref{fig:patch-coeffs-bbox}-right.
This inflation guarantees that any point $\vx$ that is near $\vP$ is contained in $B_\lbl{near}(\vP)$ and, for an admissible set of quadrature patches $\Pcoarse$, that any $\vx\in \Omega_N$ must be contained in some quadrature patch's near-zone bounding box.
This means that by forming $B_\lbl{near}(\vP)$ for each quadrature patch in $\Pfine$, a check point is in $\Omega_I$ if it is not contained in any near-zone bounding boxes.

To compute the upsampled patch set from $\Pcoarse$, we initially set $\Pfine = \Pcoarse$, compute the near-zone bounding boxes of each patch in $\Pfine$ and insert them into an \aabb tree.
We also construct the set of check points $C$ required to evaluate our discretized layer-potential with \qbkix (\cref{sec:singular-eval}).
For each check point $\vc \in C$, we query the \aabb tree for all near-zone bounding boxes that contain $\vc$.
If there are no such boxes, we know $\vc$ is far from all quadrature patches and can continue.
If, however, there are near-zone bounding boxes $B_{i_0},\hdots, B_{i_k}$ containing $\vc$, we compute the distances $d_{i_k}$ from $\vc$ to $\vP_{i_1},\hdots,\vP_{i_k}$ using \cite[Section 2]{morse2020bsupplementary}.
If $d_{i_k} < L(\vP_{i_k})$, we replace $\vP_{i_k}$ in $\Pfine$ with its four children produced by quadrisection.

To improve the performance of this refinement procedure, we allow for the option to skip the Newton method in \cref{alg:closest_point} and immediately refine all patches $\vP_{i_0},\hdots \vP_{i_k}$.
This is advantageous in the early iterations of the algorithm, when most check points are near to patches by design.
We allow for a parameter $n_\lbl{skip}$ to indicate the number of iterations to skip the Newton optimization and trigger refinement immediately.
We typically set $n_\lbl{skip}=2$.
We summarize our algorithm in \cref{alg:adaptive_upsampling}.

\begin{algorithm}[!htp]
    \KwData{An admissible patch set $\qP$, number of iterations $n_\lbl{skip}$ before using \cite[Section 2]{morse2020bsupplementary}}
  \KwResult{An upsampled set of quadrature patches}
  
  \DontPrintSemicolon
  Compute inflated near-zone bounding boxes $B_1, \hdots,  B_N$ of each $\vP \in \qP$.\;
  Construct an AABB tree $T$ from the near-zone bounding boxes.\;
  Construct all check points $C$ required to evaluate the \cref{eq:int-eq} on $\qP$.\;

  $\qP_\lbl{fine} = \qP$\;
  Mark all check points in $C$ as near.\;
  $i=0$

  \While{any $\vc \in C$ is marked near}{
    \For{$\vc \in C$}{
        \If{$\vc$ is marked near}{
            Query $T$ for all bounding boxes $B_{i_1}, \hdots B_{i_k}$ containing $\vc$.\;
            $\vP_{i_1}, \hdots \vP_{i_k} = $ patches corresponding to boxes $B_{i_1}, \hdots B_{i_k}$\;
            Mark $\vc$ as far\;
            \For{$\vP \in \vP_{i_1}, \hdots \vP_{i_k}$}{
                \eIf{$i > n_\lbl{skip}$}{
                    Find the closest point $\vector{s}_{\vc}$ on $\vP$ to $\vc$ with \cref{alg:closest_point}.\;
                    \If{ $\|\vector{s}_{\vc} - \vc\|_2 < L(\vP)$}{
                        Split $\vP$ and replace it in $\qP_\lbl{fine}$ with its children. \;
                        Mark $\vc$ as near\;
                    }
                } {
                    Split $\vP$ and replace it in $\qP_\lbl{fine}$ with its children.\;
                    Mark $\vc$ as near\;
                }
 
            }
       }
    }
    $i=i+1$\;

  }
    \mcaption{alg:adaptive_upsampling}{Adaptively upsample to accurately evaluate \cref{eq:double_layer_disc} at check points}{}
\end{algorithm}


\subsection{Marking target points for evaluation\label{app:point_marking}}

Once we have solved \cref{eq:linear_system} for $\phi$ on $\Gammah$, we need the ability to evaluate \cref{eq:double_layer} at an arbitrary set of points in the domain.
For a target point $\vx$, in order apply the algorithm in \cref{sec:singular-eval}, we need to determine whether or not $\vx \in \Omega$ and, if so, whether $\vx $ is in $\Omega_N, \Omega_I$ or $\Omega_F$.
Both of these questions can be answered by computing the closest point $\vsx$ on $\Gammah$ to $\vx$. 
If $\vn(\vsx)\cdot(\vx - \vsx) < 0$, then $\vx \in \Omega$. 
As we have seen in \cref{sec:adaptive_upsampling}, the distance $\|\vx - \vsx\|$ determines whether $\vx \in \Omega_N,\Omega_I$ or $\Omega_F$.
However, for large numbers of target points, a brute force calculation of closest points on $\Gammah$ to all target points is prohibitively expensive.
We present an accelerated algorithm combining \cref{alg:closest_point} and an \fmm evaluation to require only constant work per target point.

\subsubsection{Marking and culling far points \label{sec:mark_far}}
A severe shortcoming of \cref{alg:closest_point} is that its performance deteriorates as the distance from $\vx$ to $\Gammah$ increases. 
Consider the case where $\Gammah$ is a sphere with radius $r$ with $\vx$ at its center.
The first stage of \cref{alg:closest_point} returns a single quadrature patch that is distance $r$ from $\vx$; the next stage will return all quadrature patches.
This will take $O(N)$ time to check the distance to each patch.
Even on more typical geometries, we observe poor performance of \cref{alg:closest_point} when $\vx$ is far from $\Gammah$.

To address this, we use an additional \fmm-based acceleration step to mark most points far from $\Gammah$ before using applying \cref{alg:closest_point}. 
Our approach is based on computing the generalized winding number \cite{jacobson2013robust} of $\Gammah$ at the evaluation points. 
For closed curves in $\mathbb{R}^2$, the \textit{winding number} at a point counts the number of times the curve travels around that point. 
The \textit{generalized winding number} of a surface $\Gammah$ at a point $\vx \in \mathbb{R}^3$ can be written as 
%

\begin{equation}
  \omega_{\Gammah}(\vx) = -\frac{1}{4\pi}\int_{\Gammah} \frac{(\vx - \vy) \cdot \vn}{\|\vx - \vy\|^3} d\vy_{\Gammah}
  \label{eq:gen_winding2}
\end{equation}
We recognize this integral as the double-layer potential in \cref{eq:double_layer} for a Laplace problem with $\phi = 1$. 
Its values in $\mathbb{R}^3$ are  \cite{K}:
\begin{equation}
  \omega_{\Gammah}(x) =
  \begin{cases}
    1 & \vx \in \Omega \setminus \Gammah\\
    1/2 & \vx \in \Gammah\\
    0 &  \vx \in \mathbb{R}^3 \setminus \overline{\Omega}
  \end{cases}
  \label{eq:const-density}
\end{equation}
\cref{eq:gen_winding2} can be evaluated using the same surface quadrature in \cref{eq:double_layer_disc} using an \fmm in $O(N)$ time.
While the quadrature rule is inaccurate close to the surface, $\Omega_F$ is defined precisely as the zone where the quadrature rule is sufficiently accurate. 
For this reason, we use
\begin{equation}
  |\omega_{\Gammah}(\vx) -1| < \etrg
  \label{eq:marking-far}
\end{equation}
to mark points $\vx \in \Omega_F \subset \Omega$ and a similar relation
\begin{equation}
  |\omega_{\Gammah}(\vx)| < \etrg
  \label{eq:marking-out}
\end{equation}
to mark points $\vx \not \in \Omega$.
This approach is similar in spirit to the spectrally accurate collision detection scheme of \cite[Section 3.5] {QB}.
Unlike \cite{QB}, however, we do \textit{not} use singular integration to mark all points. 
This isn't possible since at this stage since we do not yet know which target points require singular integration. 
We use the \fmm evaluation purely as a culling mechanism before applying the full marking algorithm.

\noindent\textbf{Remark:} Since the quadrature rule may be highly inaccurate for points close to the surface, due the near-singular nature of the integrand, $\omega_{\Gammah}(\vx)$ may happen to be close to one or zero. 
We highlight that it is possible that points outside $\Omega_F$ may be mismarked, although we have not observed this in practice. 


\subsubsection{Full marking algorithm}
We combine the algorithms of the previous two sections into a single marking pipeline for a general set of target points in $\mathbb{R}^3$, by first applying the algorithm of \cref{sec:mark_far} to mark all points satisfying  \cref{eq:marking-far} then passing the remaining points to \cref{alg:closest_point}.
The full marking algorithm is summarized as \cref{alg:point_marking}.

\begin{algorithm}
  \KwData{An admissible set of quadrature patches $\qP, \etrg$, target points $\vX$}
  \KwResult{A marked set of target points $\vX$}
  
  \DontPrintSemicolon
  $\phi_0 = 1$\;
  $\omega_{\Gammah} =$ \texttt{Laplace\_FMM}($\qP$, $\vX$, $\phi_0$)\;

  \For{$\vx \in \vX$}{
    \uIf{$|\omega_{\Gammah}(\vx) - 1| < \etrg$}{
      Mark $\vx$ as inside $\Omega$.\;
      Mark $\vx$ as in $\Omega_\lbl{F}$.\;
    }
    \uElseIf{$|\omega_{\Gammah}(\vx)| < \etrg$}{
      Mark $\vx$ as outside $\Omega$.\;
    } 
  }
  \For{$\vx \in \vX$}{
    \If{$\vx$ is unmarked}{
      
        Compute the closest point $\vsx$ to $\vx$ with \cref{alg:closest_point}\;
        $d_\lbl{min} = \|\vsx -\vx\|_2$\;
        \eIf{$d_\lbl{min} \leq L_{\vsx}$}{
            Mark $\vx$ as in $\Omega_N$\;
        }{
            Mark $\vx$ as in $\Omega_I$\;
        }
        \If{$\vn(\vsx)\cdot(\vx-\vsx) < 0$}{
            Mark $\vx$ as inside $\Omega$\;
        }{
            Mark $\vx$ as outside $\Omega$\;
        }
    }
  }
  \mcaption{alg:point_marking} {Mark points in regions $\Omega_F$, $\Omega_I$ and $\Omega_N$}{}

\end{algorithm}

%% file: error.tex
\section{Error Analysis\label{sec:error}}
As with other approximation-based quadrature methods, \qbkix has two primary sources of error: the quadrature error $e_Q$ incurred as a result of evaluating potential at the check points and the extrapolation error $e_E$ due to evaluating the polynomial approximation of the potential at the target point, assuming $\Pcoarse$ is admissible. Let
\begin{align}
e_Q(\vx) &=\left|\sum_{s=0}^p (u(c_s) - \hat{u}(c_s,\Pfine) )\ell_s(t_\vx)\right|,
  \label{eq:err-quad}\\
e_E(\vx) &= \left|u(\vx) - \sum_{s=0}^p u(c_s)\ell_s(t_\vx)\right|,
  \label{eq:err-ext} \\
e_\lbl{hedgehog}(\vx) &\leq e_Q(\vx) + e_E(\vx),
  \label{eq:err-hedgehog} \\
\end{align}
where $u(\vx)$ and $\hat{u}(\vx, \Pfine)$ are defined in \cref{eq:double_layer,eq:double_layer_disc} and $\ell_s(t)$ is the $s$-th Lagrange polynomial defined on the points $\{0,1,\hdots, p\}$.
We define $t_\vx$ such that $\vx = -\vn(\vy)(R + t_\vx r)$, so $t_\vx = \frac{\|\vx - \vy\| - R}{r}$. 
In this section, we first prove that we achieve high-order accuracy with our singular/near-singular evaluation scheme in \cref{sec:singular-eval} with respect to extrapolation order $p$ and quadrature order $q$.
We then detail the impact of surface approximation on overall solution accuracy.


\subsection{Quadrature error\label{sec:quad_error}}
We briefly state a tensor-product variation of known Clenshaw-Curtis quadrature error results as applied to smooth functions in \threed.  
This estimate is derived based on assumptions detailed in \cref{app:proof_of_error_quad_high_order} that, in general, is difficult to verify in practice and may not hold for all functions we consider.
For this reason, we refer to it as a heuristic.

\begin{heuristic}
    Let the boundary $\Gammah$ be discretized by quadrature patches over the domains $[-h,h]$ and the boundary condition $f$ in \cref{eq:pde} be at least $C^k$.
  Apply the $q$-th order Clenshaw-Curtis quadrature rule to the double-layer potential $u(\vx)$ given in \cref{eq:double_layer_patches} and let $\vx$ be in the interior of $\Omega$. 
  Then for all sufficiently large $q$:
\begin{align}
  e_\lbl{Q}(\vx) &
  \lesssim  \frac{128h^{k+1}}{15\pi k(2q + 1 - k)^{k}} \tilde{V}, \\ 
  \shortintertext{where}
  \tilde{V} &= \max_{i=1,\hdots,N} \max_{\alpha,\beta\leq k}\left\| \frac{\partial^{\alpha+\beta}}{\partial u^\alpha\partial v^\beta}\left(\frac{\partial G(\vx,\vP_i(s,t))}{\partial \vn}\phi(\vP_i(s,t)) g_{\vP_i}(s,t)\right)\right\|_T,
  \label{eq:error_quad_gen_target}
\end{align}
 $g_\vP$ is the determinant of the metric tensor of a patch $\vP$ implicit in \cref{eq:double_layer_patches}, $\lesssim$ means "approximately less than or equal to," and $\|\zeta\|_T = \|\zeta^\prime/\sqrt{1-x^2}\|_1$.
  \label{heuristic:error_quad_high_order}
\end{heuristic}

This heuristic captures the qualitative behavior of the error. 
We present the derivation of \Cref{heuristic:error_quad_high_order} in \cref{app:proof_of_error_quad_high_order}.
This heuristic is  insufficient for direct application to \cref{eq:double_layer_patches}. 
As $\vx \to \Gammah$, the value of $k$ required in \Cref{heuristic:error_quad_high_order} grows rapidly due to growing higher order derivatives of the integrand. 
Such large values of $q$ and $k$ imply that smooth quadrature rules are cost-prohibitive; this is the problem that singular/near-singular quadrature schemes like \qbkix aim to address.
Moreover, this estimate is too loose to determine whether \qbkix or smooth quadrature is required to evaluate the potential.
The assumption in \cref{sec:adaptive_upsampling} addresses this problem by providing a cheap, reasonably robust criterion for refinement that is motivated by existing analyses \cite{aT2,barnett2014evaluation} instead of relying on \Cref{heuristic:error_quad_high_order}.

\subsection{Extrapolation error\label{sec:extrap_error}}
A reasonable critique of \qbkix is its reliance on an equispaced polynomial interpolant to extrapolate values of $u$ to the target point.
Despite using the first-kind barycentric interpolation formula \cite{webb2012stability}, polynomial interpolation and extrapolation in equispaced points is well-known for an exponentially growing Lebesgue constant and poor stability properties as the number of points $p$ increases \cite{trefethen1991two,platte2011impossibility}.
Recently \cite{DT} demonstrated stable extrapolation in equispaced $p+1$ points
using least-sqaures polynomials of degree $\sqrt{p}$.
However, these results are asymptotic in nature and don't tell the full story for small to moderate values of $p$, as in the \qbkix context.

\begin{figure}[!htb]
      \centering
      \includegraphics[width=.45\linewidth]{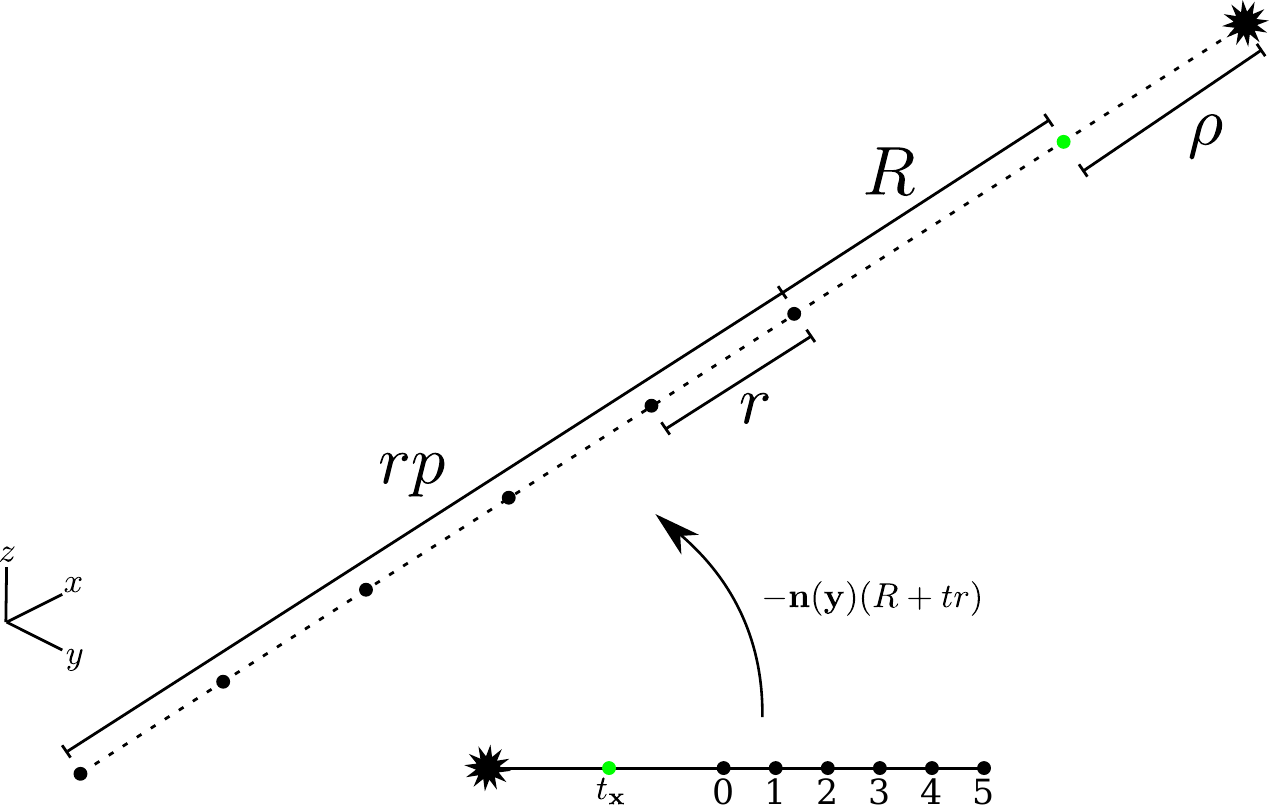}
  \mcaption{fig:extrap-err-setup}{Diagram of extrapolation setup}{ The toy setup used to study the extrapolation error of a singular function. We choose a simple point singularity $\mu(t) = \frac{1}{\|t - q\|}$ where $q = (\rho, 0, 0)$ (black star) with $\rho =-.1$.
  We choose samples at the points $t_i = (R+ir, 0,0)$ for $i=0,\hdots, p$ (black dots) and extrapolate the values $\mu(t_0),\hdots, \mu(t_p)$ to $t=0$ (green dot).
    }
\end{figure}

We begin our discussion with a simple representative experiment in equispaced extrapolation.
\Cref{fig:extrap-err-setup} depicts a minimal extrapolation setup in \threed of a simple singular function $\mu(t) = 1/\|t-q\|$ along a line, with $q = (\rho, 0, 0)$ and $\rho = -.1$.
We extrapolate exact values of $\mu$ from $p$ points, located at $t_i = (R + ir,0,0)$, to the origin.
This closely mimics the worse-case extrapolation error in \oned of a function analytic in a Bernstein ellipse with a real axis intercept of $\rho+R+ rp/2$.
We repeat this for a large range of values of $r$ and $R$ for various values of $p$. 
The log of the relative error is plotted in \Cref{fig:extrap-err-p6,fig:extrap-err-p8,fig:extrap-err-p10,fig:extrap-err-p12,fig:extrap-err-p14} as a function of the relative extrapolation interval size $rp/R$ and the scaled extrapolation distance $R/\rho$.

As mentioned in \cite[Section 3.4]{RBZ}, the adaptive refinement of $\Pcoarse$ resolves the boundary data $f$, and therefore $u$ and $\phi$, on the length scale $L$ of the patch.
This means we can reasonably assume that the distance of the nearest singularity is $O(L)$ from $\Gammah$, i.e., $\rho = \lambda L$ for some $\lambda$.
In the context of \qbkix, we know that $R=bL(P)$ and $r=aL(P)$.
\Cref{fig:extrap-err-p6,fig:extrap-err-p8,fig:extrap-err-p10,fig:extrap-err-p12,fig:extrap-err-p14}
are a study of extrapolation error as a function of $a/b$, $b/\lambda$ and $p$.

\begin{figure}[!htb]
  \setlength\figureheight{1.9in}
  \setlength\figurewidth{2.1in}
  \centering
  \hfill

  \begin{subfigure}{.33\textwidth}
      \centering
      \includegraphics[width=\linewidth]{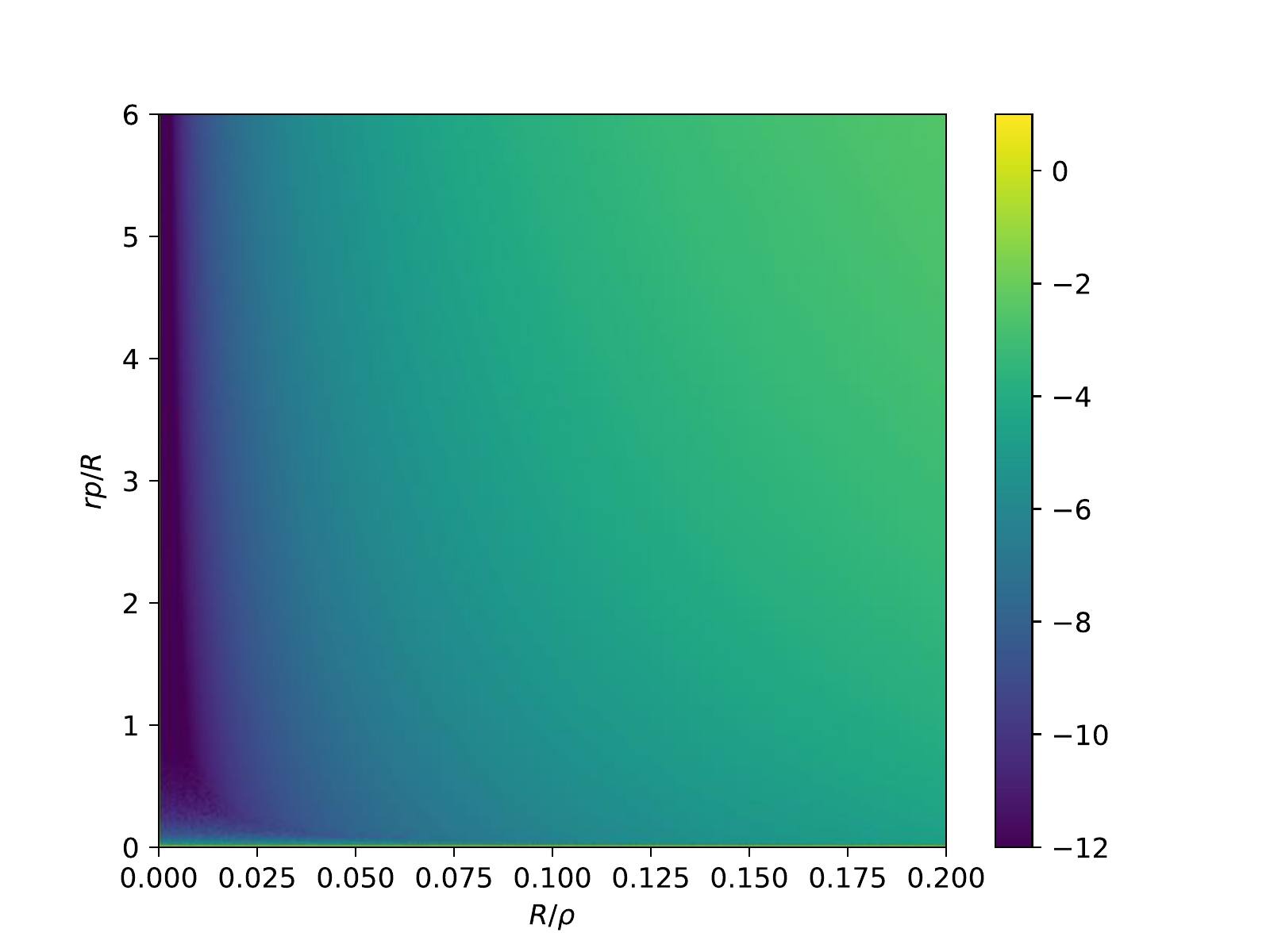}
    \caption{\label{fig:extrap-err-p6}}
  \end{subfigure}
  \begin{subfigure}{.33\textwidth}
      \centering
      \includegraphics[width=\linewidth]{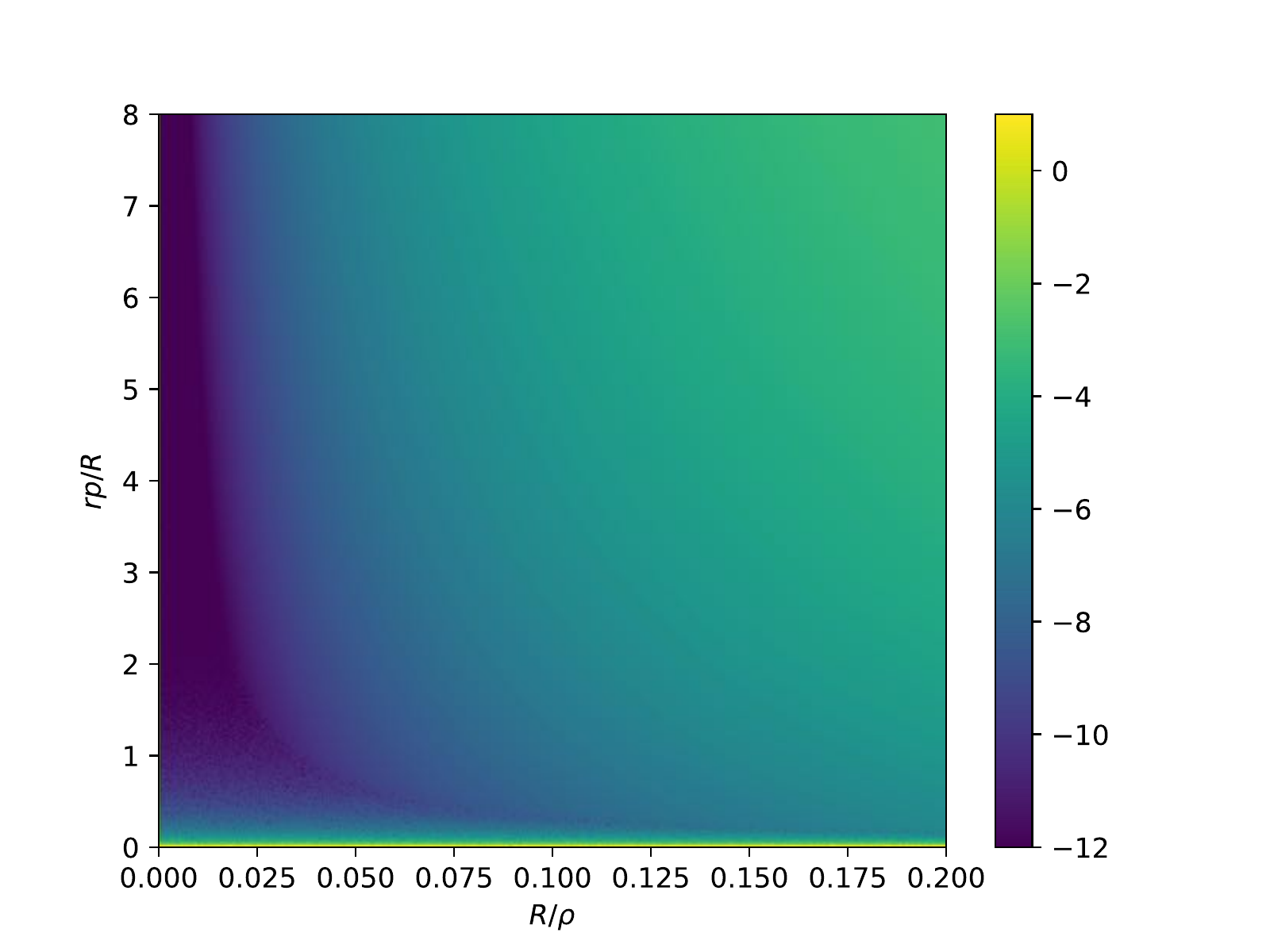}
    \caption{\label{fig:extrap-err-p8}}
  \end{subfigure}%
    \begin{subfigure}{.33\textwidth}
      \centering
      \includegraphics[width=\linewidth]{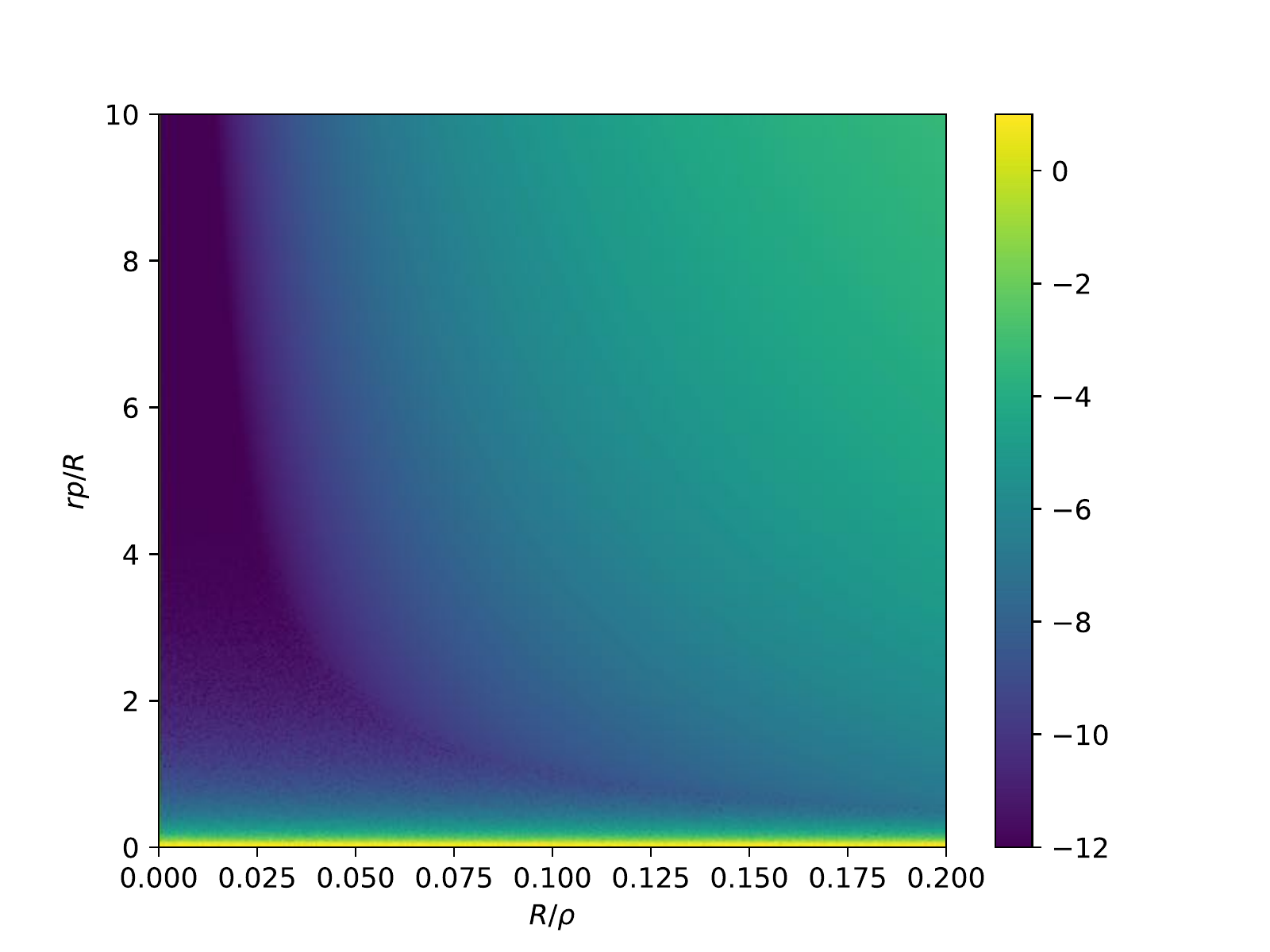}
    \caption{\label{fig:extrap-err-p10}}
  \end{subfigure}
  \begin{subfigure}{.33\textwidth}
      \centering
      \includegraphics[width=\linewidth]{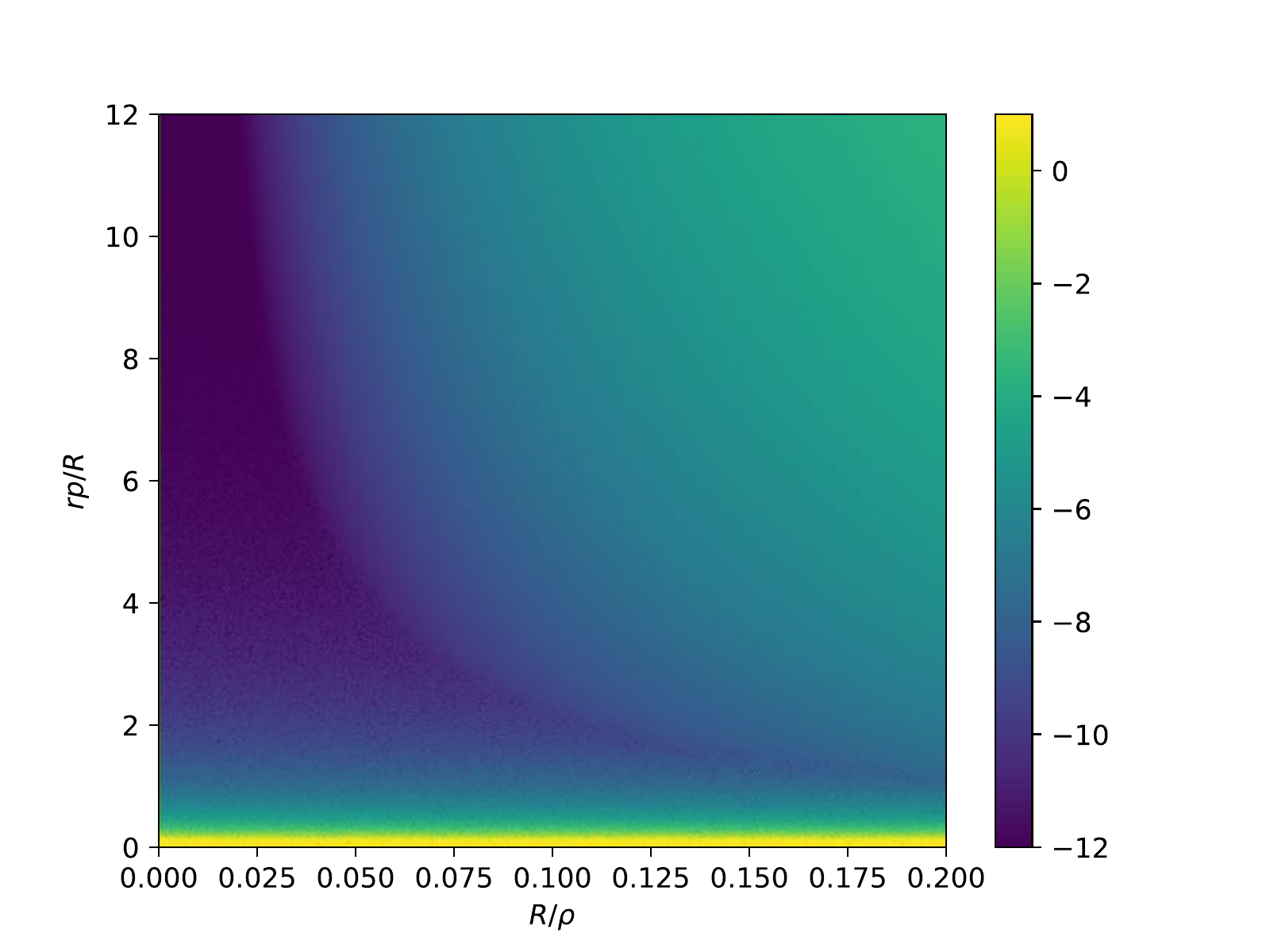}
    \caption{\label{fig:extrap-err-p12}}
  \end{subfigure}%
  \begin{subfigure}{.33\textwidth}
      \centering
      \includegraphics[width=\linewidth]{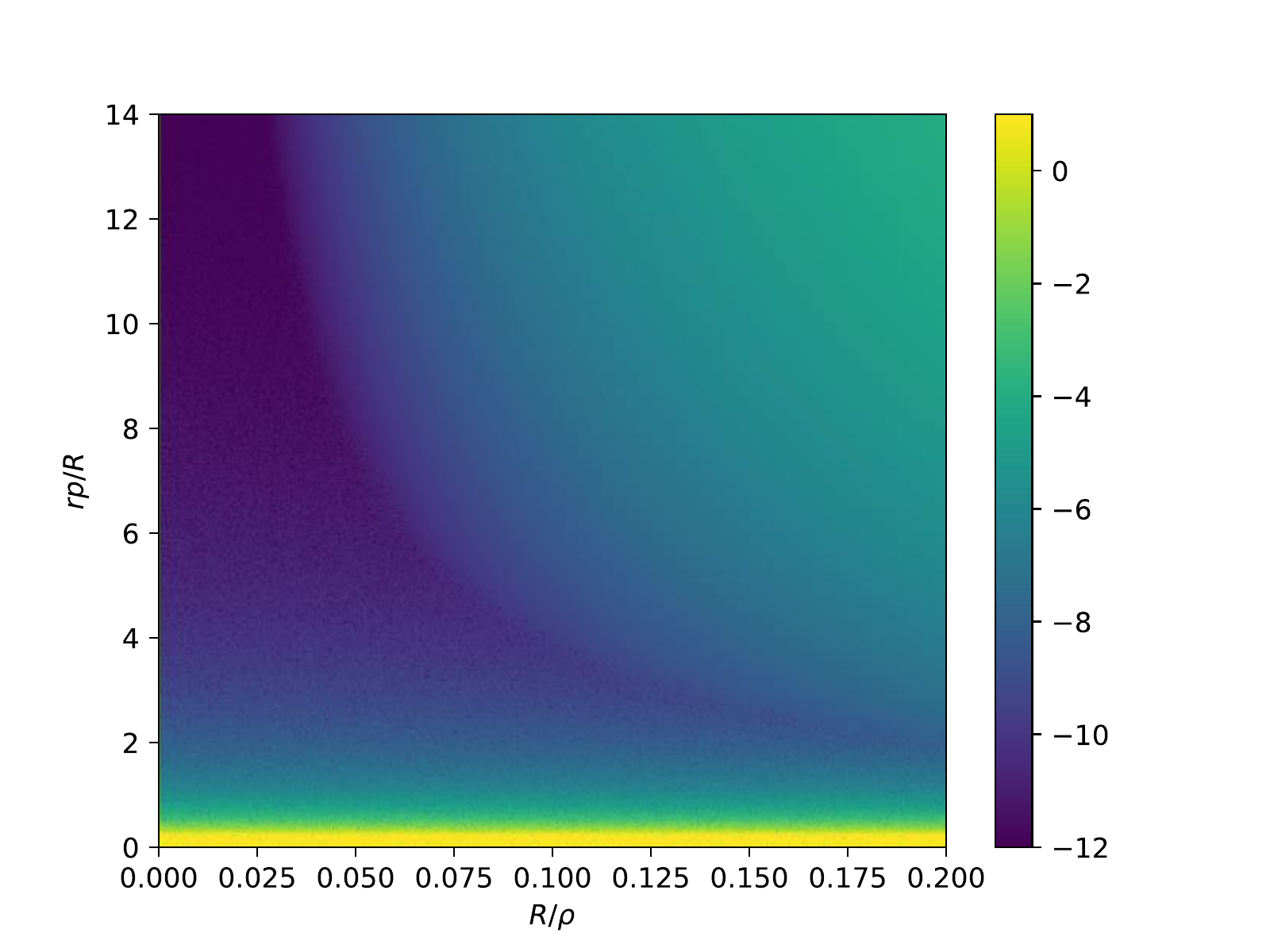}
    \caption{\label{fig:extrap-err-p14}}
  \end{subfigure}%
  \mcaption{fig:extrap-experiment}{Empirical extrapolation error behavior}{We sweep over a range of $R$ and $r$ values to vary \Cref{fig:extrap-err-setup} and plot the log of the relative error in
    \Cref{fig:extrap-err-p6,fig:extrap-err-p8,fig:extrap-err-p10,fig:extrap-err-p12,fig:extrap-err-p14}, for values $p=6,8,10,12,14$, in increasing order, from (a) to (e).
In these figures, the $x$-axis is the extrapolation distance $R$ normalized by $\rho$ and the $y$-axis is the ratio $rp/R$.
The top of the $y$-axis corresponds to $r=R$; $rp/R = 1$ corresponds to our choice of the parameter $a$.
Assuming that $\rho = O(L)$, $r/R = a/b$ and $R/\rho = b/\lambda$ for some constant $\lambda$.
}
\end{figure}
There are several important observations to make from these plots: 
\begin{itemize}
  \item Extrapolation error decreases as $R/\rho$ decreases, as expected.
  \item For a fixed value of $R/\rho$, the extrapolation error \textit{decreases} rapidly as $rp$ decreases, up to a certain value $r^*p$.
This is somewhat counterintuitive, since this means placing points closer together and extrapolating a further distance relative to $rp$.
For a fixed $p$ in exact arithmetic, letting the interpolation interval size tend to zero produces an order $p$ Taylor expansion of the solution $u$ centered at the interval's origin, which accounts for this phenomenon.


  \item Beyond $r^*p$, the extrapolation error \textit{increases}.
    The effects of finite precision eventually pollutes the convergence behavior described above. 
    Moreover, the spacing $r^*$ appears to be a function of $p$.
    For $p=6$, $r$ can be reduced to $1/p$ without any numerical issues, but
    by $p=14$, only $r>\frac{1}{2}$ is a safe choice for extrapolation.
\end{itemize}
We do not aim to rigorously analyze these phenomena in this work. 
We highlight them to provide empirical evidence that equispaced extrapolation is a reasonable, but not optimal, choice for our problem of singular/near-singular integration and to provide some intuition for our parameter choices.

The following simple result describes the behavior of the extrapolation error in \cref{eq:err-ext}.
\begin{theorem}
    Let $u(\vc(t))$ be the solution to \cref{eq:pde} given by \cref{eq:double_layer}, restricted to the line $\vc(t)$ in \threed intersecting $\vx$, let $\vc(t)$ be given by 
  \begin{equation}
    \vc(t) = \vsx - (R + tr)\vn(\vsx),
    \label{eq:check_point_param}
  \end{equation}
  where $\vsx$ is the closest point on $\Gammah$ to $\vx$, $R=bL_{\vsx}$, $r=aL_{\vsx}$, $\vn(\vsx)$ is the outward surface normal at $\vsx$, and let $|u^{(p)}(\vc(t))|$ be bounded above by $C_p$ on the interval $[-R, R+pr]$.
  Let $\mathfrak{P}(t)$ be the $p$-th order polynomial interpolant of $u(\vc(t))$ constructed from the check points $\vc_0, \hdots, \vc_p$, where $\vc_i = \vc(i)$. Then the extrapolation error associated with \qbkix behaves according to:
\begin{equation}
  |u(\vc(t_\vx)) - \mathfrak{P}(t_\vx)| \leq \frac{C_p}{(p+1)!}|R + rp|^p = \frac{C_p}{(p+1)!}|b+ap|^p\cdot | L|^p,
  \label{eq:extrap_err_init}
\end{equation}
where $t_\vx = \frac{\|\vx - \vsx\| - R}{r}$.
\label{thm:extrap_error}
\end{theorem}

\begin{proof}

    We know that for a smooth function $f$ and points $x_0, \hdots x_p$ in a \oned interval $I_0$, for some $\xi \in I_0$, the following relation holds for all $x \in I_0$:
\begin{equation}
    f(x) - \mathfrak{P}(x) = \frac{f^{(p)}(\xi)}{(p+1)!}\prod_{i=0}^p(x-x_i).
  \label{eq:exterp_err_init}
\end{equation}
Let $\mathfrak{P}$ be the $p$th order polynomial interpolating the points $x_0,\hdots x_p$.
In the \qbkix setup, since $R+rp$ is the distance of the furthest check point to $\vy$, we know that $x - x_i < R +rp$ for each $i$.
Since $f(t) = u(\vc(t))$ is harmonic, and therefore $C^\infty$, in $\Omega$, $|f^{(p)}(\xi)|$ can be uniformly bounded on $I_0$ by some constant $C_p$,
Noting that $R = bL$ and $r=aL$ yields our result.
\end{proof}

For fixed values of $a$ and $b$, as we let $L\to 0$, the extrapolation error is bounded by $O(L^p)$.
In practice, however, this means that we can choose $a$ and $b$ to minimize the constant factor $|b+ap|^p$ in \cref{thm:extrap_error}.
Since $p>1$, $a$ must be chosen to balance out the contribution of $p$, yet our extrapolation study shows that we can't simply set $a=0$.  
We therefore choose $a \leq 1/p$ for $p=6$ and 8, motivated by \cref{fig:extrap-err-p6,fig:extrap-err-p8}.
Moreover, since $b < 1$, we can choose $a \leq b/p$, which allows $a$ and $b$ to decay at the same rate.
The advantage of choosing $a \leq b/p$ is that $b$ is a single parameter that controls the accuracy of \qbkix.
Since we have fixed the quadrature order $q=20$ to satisfy the assumption in \cref{sec:adaptive_upsampling}, a smaller value of $b$ will trigger more upsampling in \cref{alg:adaptive_upsampling}, keeping quadrature error fixed while reducing extrapolation error.

It is important to keep in mind that \cref{thm:extrap_error} only provides insight for moderate values of $p$; our conclusions are largely irrelevant for large $p$.
We use  $p = 6$ and $a = b/6$, leaving the construction of an optimal extrapolation extrapolation scheme to future work. 
\subsection{Limitations \label{sec:limitations}}
Our error discussion reveals several limitations of our method. 
The first and most apparent shortcoming is that extrapolation instability fundamentally limits convergence order. 
However, for reasonable orders of convergence, up to 14, we have discussed an empirical scheme to choose parameters to maximize the available convergence behavior.
Moreover, low-order surface geometries used in engineering applications will likely limit the convergence rate before it is limited by the extrapolation order, making this a non-issue in practical scenarios.

Another downside of the chosen extrapolation approach is lack of direct extension of \qbkix to oscillatory problems like the Helmholtz equation.
Due to the limitation on the values of $p$, we can't guarantee the ability to resolve high-frequency oscillations in the solution.
A new extrapolation procedure is required to do so robustly without compromising efficiency.

In \cite{wala20193d}, the authors demonstrate a relationship between the truncation error of a \qbx expansion and the local curvature of $\Gammah$. 
Our scheme also is susceptible to this form of error and we do not address nor analyze this in this work.
This is a subtle problem that requires a detailed analysis of the surface geometry with respect to the chosen extrapolation scheme.
Another  limitation is the lack of an accurate error estimate to serve as an upsampling criteria in place of the criteria in \cref{sec:adaptive_upsampling}, such as \cite{klinteberg2019accurate}. 
Extending \cite{klinteberg2019accurate} to \threed surfaces is non-trivial and whether the size of $\Pfine$ would be reduced enough to outweigh the added cost of the additional Newton iterations required by their scheme remains to be seen.

Finally, for certain accuracy targets and geometries, the algorithm above may lead to an impractically high number of patches in $\Pcoarse$ and $\Pfine$. 
Geometries with nearly-touching non-local regions, as shown in \cref{fig:torii}, will see large amounts of refinement.
If the nearly-touching embeddings $\gamma_r$ are close enough, i.e., less than $10^{-10}$ apart, there is little hope of an accurate solution with a fixed computational budget.
We allow the user to enforce a minimal patch size $L_\lbl{min}$, limiting the time and memory consumption at the expense of not reaching the requested target accuracy.

%% file: complexity.tex
\section{Complexity \label{sec:complexity}}
In this section, we summarize the complexity of the algorithms required by \qbkix. 
We present a detailed complexity analysis in \cite[Section 3]{morse2020bsupplementary}.
The input to our overall algorithm is a domain boundary $\Gamma$ with $\Ninit$ patches and boundary condition $f$.
The parameters that directly impact complexity are:

\begin{itemize}
\item The number of patches $N$ \emph{after} admissibility refinement.  
  This is a function of $\Ninit$, the geometry of $\Gamma$, the definition of $f$, and the choices of parameters $a$ and $b$ in check point construction. 
\item Quadrature order $q$ and the degree of smoothness $k$ of $\Gamma$ and $f$.
  We assume that $k$ is sufficiently high to obtain optimal error behavior for a given $q$ by letting $k=2q$ in \cref{eq:error_quad_gen_target}.
\item \qbkix interpolation order $p$. 
\item The numbers of evaluation points in different zones $\Nf$, $\Ni$, and $\Nn$, with $\Nt = \Nf+\Ni+\Nn$.
\end{itemize}

The complexity is also affected by the geometric characteristics of $\Gamma$ as described 
in \cite[Section 3]{morse2020bsupplementary}.

\begin{itemize}
\item \emph{Admissibility.} The complexity of this step is $O(\Ninit \log \Ninit)$,
  with constants dependent on $\alpha_0$, $\beta_0$ and $C_J$. 
  The logarithmic factor is due to use of an \aabb tree for closest surface point queries. 

\item \emph{Upsampling.}  The complexity of upsampling is $ O(\hat{m}  N \log(N))$, where $\hat{m}$ is the largest upsampling ratio.
  The logarithmic factor appears for similar reason to admissibility,
  with constants that depend on geometric parameters and the boundary condition through the error estimate of \cref{sec:error}.
        We show that the upsampling ratio is independent of $N$ in \cite[Section 3]{morse2020bsupplementary}.
  
\item \emph{Point marking.}  Identifying which zone an evaluation point belongs to ($\Omega_F, \Omega_I$ or $\Omega_N$) depends on $N$ and the total number of points to be classified $\Nt = \Nf + \Ni + \Nn$. 
  The complexity is $O(\Nt \log N)$ with constants dependent on geometric parameters, due to the cost of closest surface point queries. 
  
\item\emph{ Far, intermediate and near zone integral evaluation.}  The complexity of these components depends on $N$ and  $\Nf$, $\Ni$ and $\Nn$ respectively, with the general form $O(s_1 N + s_2 \mathcal{N}')$, where $\mathcal{N}'$ is the number of evaluation points in the corresponding class.  For the far field, $s_1 = s_2 = 1$.  For the intermediate evaluation,
    $s_1 = \hat{m} q^2$ and $s_2 =1$; finally, for the near zone, $s_1 = \hat{m}q^2$ and $s_2 = p$.
If $b$ is chosen appropriately, the intermediate and near zone error is $\etrg$. 

  \item \emph{\gmres solve.} Due to the favorable conditioning of the double-layer formulation in \cref{eq:int-eq}, \gmres converges rapidly to a solution in a constant number of iterations for a given $\Gamma$ that is independent of $N$. 
    This means that the complexity to solve \cref{eq:int-eq} is asymptotically equal (up to a constant dependent on $\Gamma$) to the complexity equal to a near-zone evaluation with $\Nn=N(q+1)^2$. 
\item \emph{Evaluation on uniform point distribution} In many applications, one would like the value of the solution $u$ due to a density $\phi$ at a collection of points uniformly distributed throughout the domain $\Omega$.
    When the number of such targets is chosen to match the resolution of the surface discretization, the overall complexity of solution evaluation is $O((\hat{m}+\hat{m}q^2)q^2N + N^{3/2})$.
\end{itemize}

%% file: results.tex
\section{Results\label{sec:results}}
We now demonstrate the accuracy and performance of \qbkix to evaluate singular/near-singular layer potentials on various complex geometries to solve the integral equation in \cref{eq:int-eq} and evaluate the solution as defined in \cref{eq:double_layer}.

\subsection{Classical convergence with patch refinement}
We will first demonstrate the numerical convergence behavior of \qbkix.
As discussed in \cite[Section 3.1]{KBGN}, approximation-based schemes such as \qbkix do not converge classically but do so up to a controlled precision if $r$ and $R$ scale with proportional to the patch size. 
In order to observe classical convergence as we refine $\Pcoarse$, we must allow $R$ and $r$ to decrease slower than $O(L)$, such as with rate $O(\sqrt{L})$.
In this section, we choose the \qbkix parameters $a$ and $b$ proportional to $1/\sqrt{L}$ to achieve this and demonstrate numerical convergence with refinement of $L$.

In our examples, we use analytic solutions to \cref{eq:pde} obtained as sums of point charge functions of the form 
\begin{equation}
  u_c(\vx) = \sum_{i=1}^m G(\vx,\vy_i)\psi_i
  \label{eq:point-charge-solution}
\end{equation}
where the charge locations $\vy_i$ with strengths $\psi_i$ are outside of $\Omega$.
To construct specific solutions, we sample a sphere of radius one with point charges, as shown in \Cref{fig:greens-id-test-cases,fig:solver-conv-test-cases}.
We choose charge strengths $\psi_i$ randomly from $[0,1]^d$, where $d=1$ for Laplace problems and $d=3$ for Stokes and elasticity problems. 

We use the multipole order  $m=20$ with $5000$ points per leaf box for the  kernel-independent \fmm. 
This ensures that the \fmm error does not dominate;  sufficiently large number of points per leaf box is needed to minimize the additional error due to tree depth.
We choose a high quadrature order $q=20$, or 400 quadrature points per patch in $\Pcoarse$, relative to overall convergence order to satisfy the assumption in \cref{sec:adaptive_upsampling}.
We also use two levels of uniform upsampling to demonstrate convergence.

\subsubsection{Green's Identity}
\begin{figure}[!htb]
  \centering
  \setlength\figureheight{1.9in}
  \setlength\figurewidth{2.1in}
  \begin{minipage}{.5\textwidth}
      \centering
    \includegraphics[width=.75\linewidth]{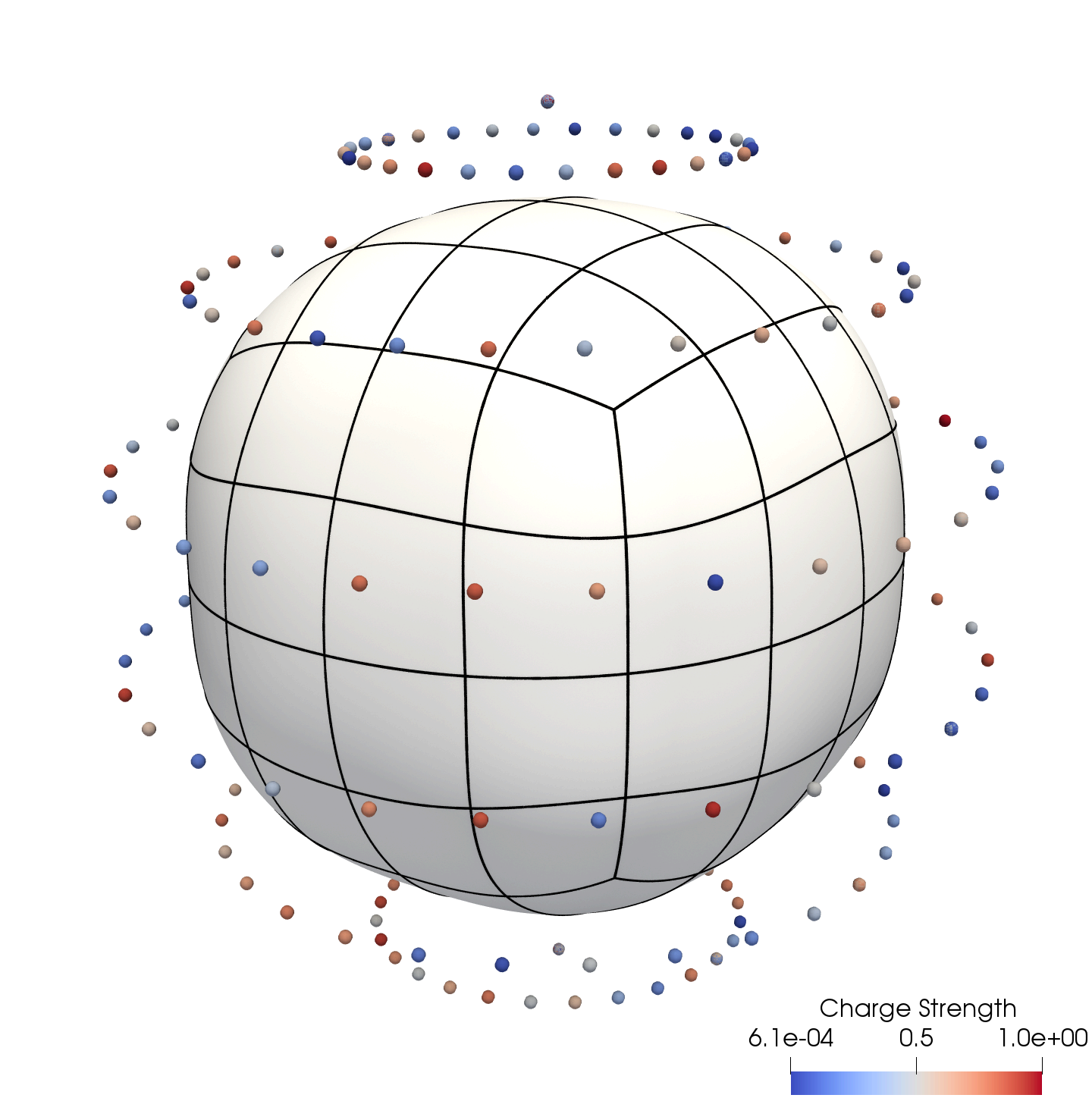}
  \end{minipage}\hfill
  \begin{minipage}{.5\textwidth}
      \centering
    \includegraphics[width=.75\linewidth]{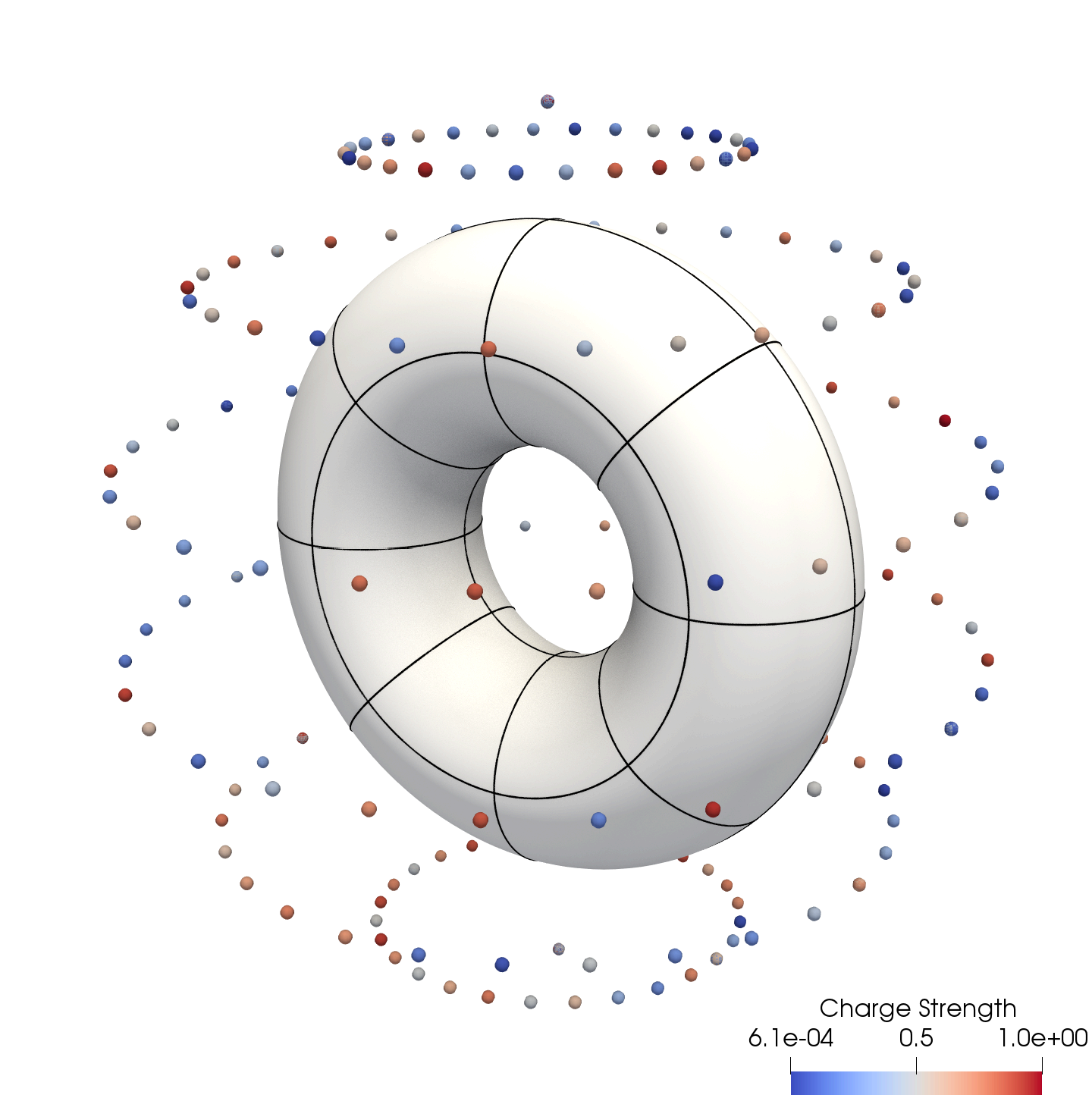}
  \end{minipage}\hfill
  \mcaption{fig:greens-id-test-cases}{}{Geometry and singularities used for Green's Identity convergence tests. 
  Shown are polynomial patches defining boundary geometry (black lines) and point singularities
  placed on the surface on a sphere of radius one.
  Singularity strengths are randomly selected values in $[0,1]$; shown is the strength intensity for Laplace problems, which varies from blue to red.
  We use 96 20th-order polynomial patches for the spheroid (left) and 32 cubic patches for the torus (right).
}
\end{figure}

\begin{table}[!htb]
\centering
\small
\setlength\tabcolsep{4pt}
\begin{tabular}{lllllll}
\toprule
{}Geometry & \pde &\multicolumn{4}{c}{Relative $\ell^\infty$ error (Number of patches)} & EOC\\
\midrule
Spheroid                                    & Laplace         &  $1.06\times 10^{-4}$ (96) &  $4.78\times 10^{-6}$ (384) &  $9.14\times 10^{-8}$ (1536) &  $4.35\times 10^{-9}$ (6144) & 4.77\\
(\cref{fig:greens-id-test-cases}-left)  & Elasticity      &  $1.68\times 10^{-3}$ (96) &  $6.94\times 10^{-5}$ (384) &  $1.53\times 10^{-6}$ (1536) &  $1.33\times 10^{-8}$ (6144) & 5.74\\
                                        & Stokes          &  $1.92\times 10^{-3}$ (96) &  $7.95\times 10^{-5}$ (384) &  $1.74\times 10^{-6}$ (1536) &  $1.53\times 10^{-8}$ (6144) & 5.72 \\
 \midrule
Torus                                         & Laplace         &  $2.05\times 10^{-3}$ (32) &  $7.52\times 10^{-5}$ (128) &   $3.79\times 10^{-6}$ (512) &  $8.48\times 10^{-8}$ (2048) & 5.45\\
 (\cref{fig:greens-id-test-cases}-right)      & Elasticity      &  $4.38\times 10^{-2}$ (32) &  $1.17\times 10^{-3}$ (128) &   $5.08\times 10^{-5}$ (512) &  $1.42\times 10^{-6}$ (2048) & 5.09\\
                                              & Stokes          &  $5.03\times 10^{-2}$ (32) &  $1.33\times 10^{-3}$ (128) &   $5.81\times 10^{-5}$ (512) &  $1.65\times 10^{-6}$ (2048) & 5.09\\
\bottomrule
\end{tabular}
\mcaption{table:greens-id-data}{$\ell^\infty$ Relative error in Green's Identity versus number of patches}{
  The solution to \cref{eq:pde} due to a known function $u_c$, shown in \cref{fig:greens-id-test-cases} is computed via Green's Identity.
  We evaluate the single- and double-layer potentials with \qbkix due to the Dirichlet and Neumann boundary data and compare against the known value of $u_c$ on the boundary.
  Each column is the result of an additional level of uniform quadrisection of the patches in $\Pcoarse$. 
  The final column (EOC) is the estimated convergence order, computed via least-squares log-log fit of the error as a function of max patch size.
}
\end{table}

We report the accuracy of the \qbkix evaluation scheme in \cref{table:greens-id-data}, where we verify Green's Identity for a random known function $u_c$ in \cref{eq:point-charge-solution}.
We evaluate the Dirichlet and Neumann boundary data due to $u_c$ at the discretization points of $\Gammah$ and use one-sided \qbkix to evaluate the corresponding single- and double-layer potentials at the same discretization points.
With each column of \cref{table:greens-id-data}, we subdivide $\Pcoarse$ to more accurately resolve the boundary condition.
The error shown in \cref{table:greens-id-data} is the $\ell^\infty$-relative error in the solution value
\begin{equation}
  \frac{\left\|\hat{S}\left[\frac{\partial u_c}{\partial \vn}\right](\vx) - \hat{D} \left[ u_c \right](\vx) - u_c(\vx)\right\|_\infty}{\|u_c\|_\infty},
\end{equation}
where $\hat{S}$ and $\hat{D}$ are the single- and double-layer singular integral operators discretized and evaluated with \qbkix.
In these tests, we choose $p=6$, $r=.004\sqrt{L}$ ($a=.004/\sqrt{L}$) and $R=.03\sqrt{L}$ ($b=.03/\sqrt{L}$).
We observe roughly $5$th order convergence on both the spheroid and torus test geometries in \cref{fig:greens-id-test-cases} for each of the tested \pde's.
In \cref{table:greens-id-perf-data}, we present the number of target points evaluated per second per core with one-sided \qbkix. 
We see that performance is best for Laplace and worst for elasticity problems, as expected. 
\begin{table}[!htb]
\centering
\small
\setlength\tabcolsep{8pt}
\begin{tabular}{llllll}
\toprule
{}Geometry & \pde &\multicolumn{4}{c}{Target points/second/core}\\
\midrule
Spheroid                                    & Laplace         &  $3684$ &  $5438$ &  $5077$ &  $5629$ \\
(\cref{fig:greens-id-test-cases}-left)  & Elasticity      &  $1325$ &  $1731$ &  $1687$ &  $1790$ \\
                                        & Stokes          &  $1635$ &  $2075$ &  $2016$ &  $2120$ \\
 \midrule
Torus                                         & Laplace         &  $2729$  &  $3373$ &   $4564$  &  $5477$ \\
 (\cref{fig:greens-id-test-cases}-right)      & Elasticity      &  $984$   &  $1171$ &   $1347$  &  $1502$ \\
                                              & Stokes          &  $1134$  &  $1331$ &   $1609$  &  $1727$ \\
\bottomrule
\end{tabular}
\mcaption{table:greens-id-perf-data}{Performance of singular evaluation in Green's Identity}{
    For each test in \cref{table:greens-id-data}, we report the number of target points evaluated with one-sided \qbkix per second per core.
}
\end{table}

\subsubsection{Solution via \gmres}

\begin{figure}[!htb]
  \centering
  \setlength\figureheight{1.9in}
  \setlength\figurewidth{2.1in}
  \hfill
  \begin{minipage}{.5\textwidth}
      \centering
    \includegraphics[width=.75\linewidth]{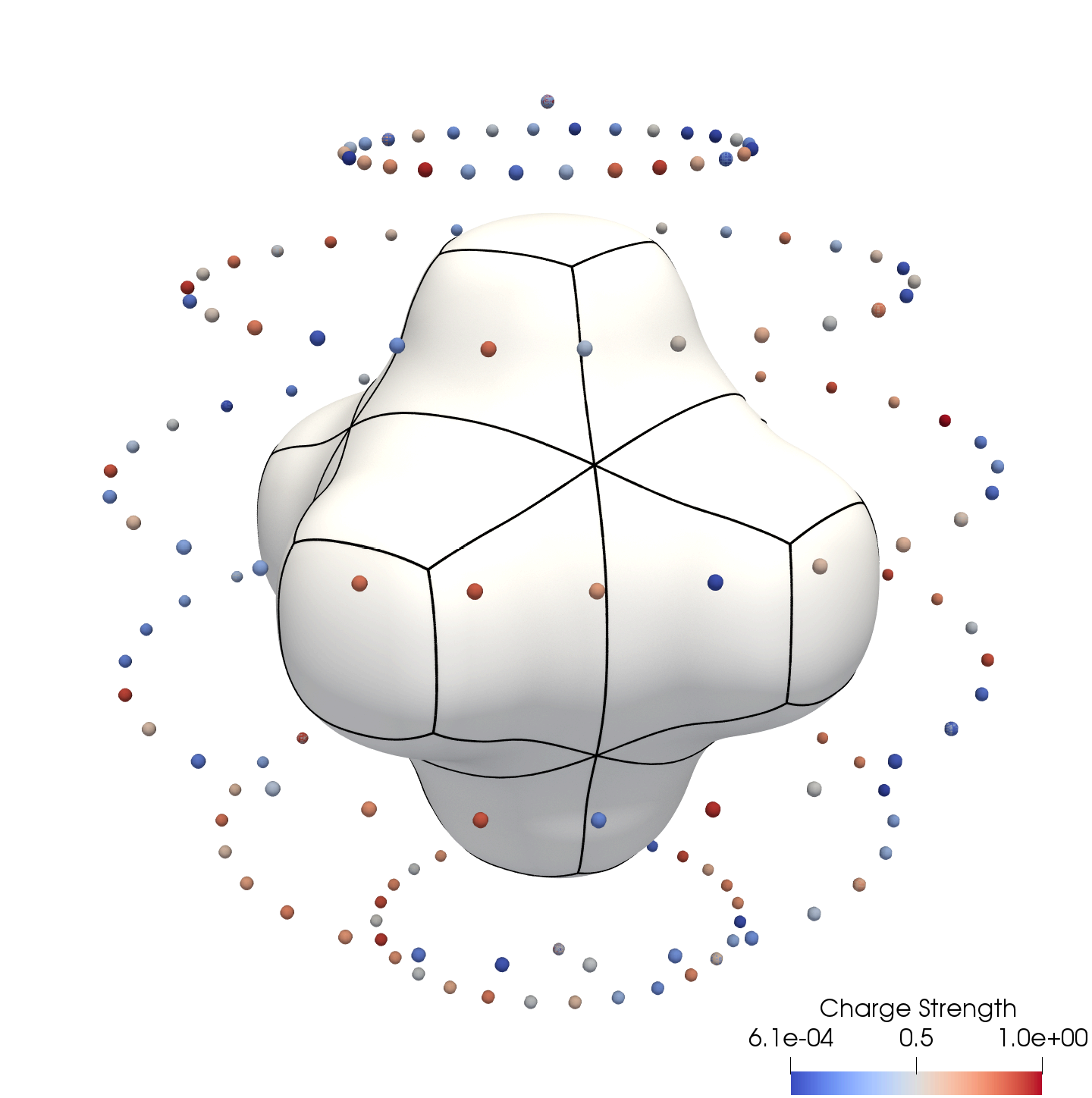}
  \end{minipage}\hfill
  \begin{minipage}{.5\textwidth}
      \centering
    \includegraphics[width=.75\linewidth]{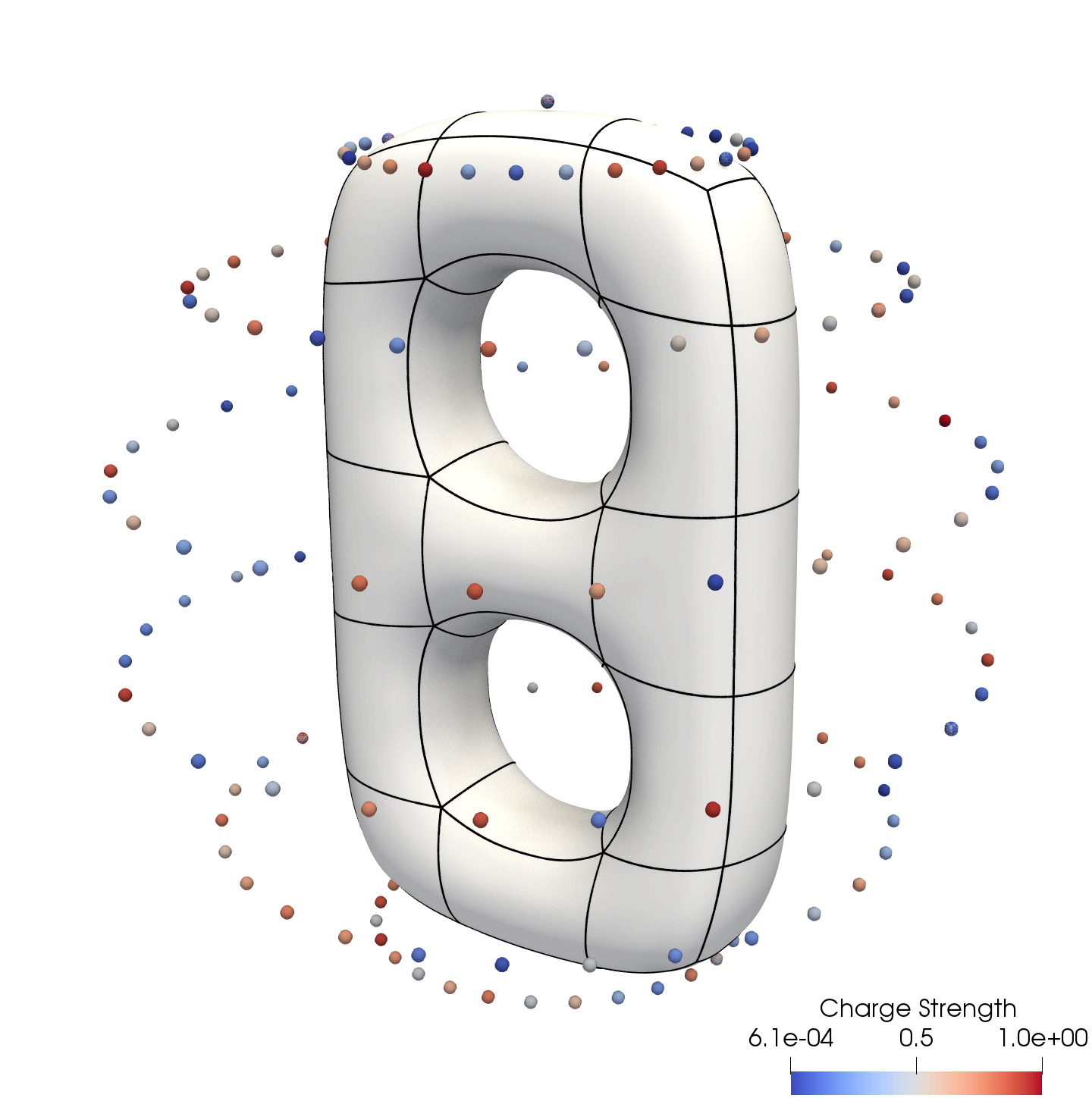}
  \end{minipage}\hfill
  \mcaption{fig:solver-conv-test-cases}{}{Geometry and singularities used for solver convergence tests. Figures are similar to \cref{fig:greens-id-test-cases}, but displaying geometries for testing the convergence of \qbkix within a \gmres solver.
  We use 30 16th-order polynomial patches for the pipe (left) and 50 20th-order patches for the genus two surface (right). 
  Note the proximity of the singularities to the domain of the genus two surface; the nearest singularity is less than $.05L$ from $\Gammah$.
}
\end{figure}

\begin{table}
\centering
\small
\setlength\tabcolsep{4pt}
\begin{tabular}{lllllll}
\toprule
{}Geometry & \pde &\multicolumn{4}{c}{Relative $\ell^\infty$ error (Number of patches)} & EOC\\
\midrule
Spheroid (\cref{fig:greens-id-test-cases}-left) & Laplace   &  $2.70\times 10^{-6 }$ (96) &  $1.92\times 10^{-7 }$ (384) &  $4.47\times 10^{-9 }$ (1536) &  $5.13\times 10^{-11 }$ (6144) & 5.35 \\
   \midrule
Pipe                                       & Laplace       &  $5.99\times 10^{-4 }$ (30) &  $3.03\times 10^{-5 }$ (120) &   $6.68\times 10^{-7 }$ (480) &  $2.27\times 10^{-8 }$ (1920) & 5.92\\
(\cref{fig:solver-conv-test-cases}-left)   & Elasticity    &  $7.17\times 10^{-2 }$ (30) &  $3.57\times 10^{-3 }$ (120) &   $8.90\times 10^{-5 }$ (480) &  $4.14\times 10^{-6 }$ (1920) & 5.45\\
                                           & Stokes        &  $8.53\times 10^{-2 }$ (30) &  $4.12\times 10^{-3 }$ (120) &   $1.03\times 10^{-4 }$ (480) &  $4.73\times 10^{-6 }$ (1920) & 5.43\\
   \midrule
Genus 2                                     & Laplace     &  $4.00\times 10^{-2 }$ (50) &  $1.25\times 10^{-4 }$ (200) &   $1.54\times 10^{-6 }$ (800) &  $5.73\times 10^{-10 }$ (3200) & 8.76 \\
(\cref{fig:solver-conv-test-cases}-right)   & Elasticity  &  $9.20\times 10^{-2 }$ (50) &  $1.05\times 10^{-3 }$ (200) &   $1.00\times 10^{-5 }$ (800) &  $9.44\times 10^{-8 }$ (3200) & 6.89\\
                                            & Stokes      &  $1.03\times 10^{-1 }$ (50) &  $1.18\times 10^{-3 }$ (200) &   $1.15\times 10^{-5 }$ (800) &  $1.03\times 10^{-7}$ (3200)& 6.88\\
\bottomrule
\end{tabular}
\mcaption{table:solver-data}{$\ell^\infty$ Relative error in \gmres solve and solution evaluation versus number of patches}{
  We solve \cref{eq:pde} by discretizing and evaluating the layer potential in the integral equation in \cref{eq:int-eq} as described in \cref{sec:singular-eval}.
  We use two-sided \qbkix inside of \gmres to solve for $\phi$, then evaluate \cref{eq:double_layer_disc} with one-sided \qbkix at a new set of points on $\Gammah$.
  Each column is the result of an additional level of uniform quadrisection of the patches in $\Pcoarse$. 
  The final column (EOC) is the estimated convergence order, computed via least-squares log-log fit of the error as a function of max patch size.
}
\end{table}

\begin{table}[!htb]
\centering
\small
\setlength\tabcolsep{8pt}
\begin{tabular}{llllll}
\toprule
{}Geometry & \pde &\multicolumn{4}{c}{Target points/second/core}\\
\midrule
Spheroid                                    & Laplace         &  $2737$ &  $3149$ &  $2846$ &  $2950$ \\
\midrule
Pipe                                    & Laplace         &  $3046$ &  $2178$ &   $2832$  &  $2982$ \\
(\cref{fig:greens-id-test-cases}-left)  & Elasticity      &  $991$  &  $993$  &   $1189$  &  $1261$  \\
                                        & Stokes          &  $1048$ &  $1140$ &   $1335$  &  $1422$  \\
 \midrule
Genus 2                                       & Laplace         &  $1862$ &  $2886$ &  $3122$  &  $2879$ \\
 (\cref{fig:greens-id-test-cases}-right)      & Elasticity      &  $729$  &  $1125$ &  $1255$  &  $1295$ \\
                                              & Stokes          &  $929$  &  $1304$ &  $1450$  &  $1504$ \\
\bottomrule
\end{tabular}
\mcaption{table:solver-perf-data}{Performance of singular evaluation in \gmres matrix-vector multiply}{
    For each test in \cref{table:solver-data}, we report the number of target points per second per core evaluated with two-sided \qbkix in a single \gmres matrix-vector multiplication.
}
\end{table}
We report the accuracy of the \qbkix scheme when used to solve \cref{eq:pde} via the integral equation in \cref{eq:int-eq}.
Two-sided \qbkix is used in the matrix-vector multiply inside \gmres to solve \cref{eq:int-eq} for the values of the density $\phi$ at the discretization points.
Then one-sided \qbkix is used to evaluate \cref{eq:double_layer_disc} at a slightly coarser discretization. 
Since \gmres minimizes the residual at the original discretization of \cref{eq:int-eq}, this final step prevents an artificially accurate solution by changing discretizations.
\cref{table:solver-data} lists the $\ell^\infty$ relative error values for the total solve and evaluation steps using \cref{sec:singular-eval} as we refine $\Pcoarse$ by subdivision as in the previous section.
In these tests, we choose $p=6$, $r= .005\sqrt{L}$ ($a = .005/\sqrt{L}$), and  $R= .03\sqrt{L}$ ($b = .03/\sqrt{L}$).
As for previous examples, we observe at least $5$th order convergence on all tested geometries in \cref{fig:solver-conv-test-cases} and \cref{fig:greens-id-test-cases}-left and all \pde's.
We include the spheroid example as an additional demonstration of a high accuracy solution via \gmres with our approach.
We report the number of target points evaluated per second per core with two-sided \qbkix in \cref{table:solver-perf-data}.
The results are similar to \cref{table:greens-id-perf-data}; the slower performance is because evaluation via two-sided \qbkix is more expensive than one-sided \qbkix.
\subsection{Comparison with \cite{YBZ}\label{sec:results-compare}}
In this section,  we compare our method to \cite{YBZ}, a previously proposed high-order, kernel-independent singular quadrature method in \threed for complex geometries. These characteristics are similar to  \qbkix shares these characteristics. \cite[Section 4]{morse2020bsupplementary} presents additional comparisons.


The metric we are interested is \textit{cost for a given relative error}.
Assuming the surface discretization is $O(N)$, we measure the cost of a method as its total wall time during execution $T$ divided by the total wall time of an \fmm evaluation on the same $O(N)$ discretization, $T_\lbl{FMM}$. 
By normalizing by the \fmm evaluation cost, we minimize the dependence of the cost on machine- and implementation-dependent machine-dependent parameters.

We run the tests in this section on the spheroid geometry shown in \cref{fig:greens-id-test-cases}-left.
We focus on the singular quadrature scheme of \cite{YBZ}. 
The near-singular quadrature of \cite{YBZ} is algorithmically similar to \qbkix, but since an expensive singular quadrature rule is used as a part of near-singular evaluation, it has a higher total cost.
As a result, the accuracy and cost of near-singular evaluation of \cite{YBZ} is bounded by the accuracy and cost of the singular integration scheme.

To compare the full \qbkix method with \cite{YBZ}, we fit polynomial patches to the $C^\infty$ surface of \cite{ying2004simple}, denoted $\Gamma_b$, to produce $\Gammah$ during the first step of \cref{sec:admissible_algo}.
We apply the remaining geometry preprocessing algorithms of \cref{sec:admissible_algo} to $\Gammah$ to produce $\Pcoarse$.
After producing $\Pfine$ with two levels of uniform upsampling, we solve \cref{eq:int-eq} with two-sided \qbkix on $\Gammah$ and evaluate the solution on the boundary with one-sided \qbkix.
We then solve for the solution to \cref{eq:int-eq} on $\Gamma_b$ using \cite{YBZ}. 

For each of the tests in this section, we choose some initial spacing parameter
$h_0$ to discretize the surface of \cite{ying2004simple}, as in \cite{YBZ}, and use the $16\times$ upsampled grid and floating partition of unity radius proportional to $O(\sqrt{h})$, as in the original work.
We apply \qbkix to $\Gammah$ and the scheme of \cite{YBZ} to $\Gamma_b$ with spacing $h_0/2^i$, for $i=1,\hdots 4$.

As in the previous section, we choose the parameters $r$ and $R$ of \qbkix to be $O(\sqrt{L})$. 
For both quadrature methods, we use a multipole order of $16$ for \pvfmm with at most 250 points in each leaf box.
\begin{figure}[!htb]
  \centering
  \setlength\figureheight{1.9in}
  \setlength\figurewidth{2.1in}
    \makebox[\textwidth][c]{ \includegraphics[width=1.1\textwidth]{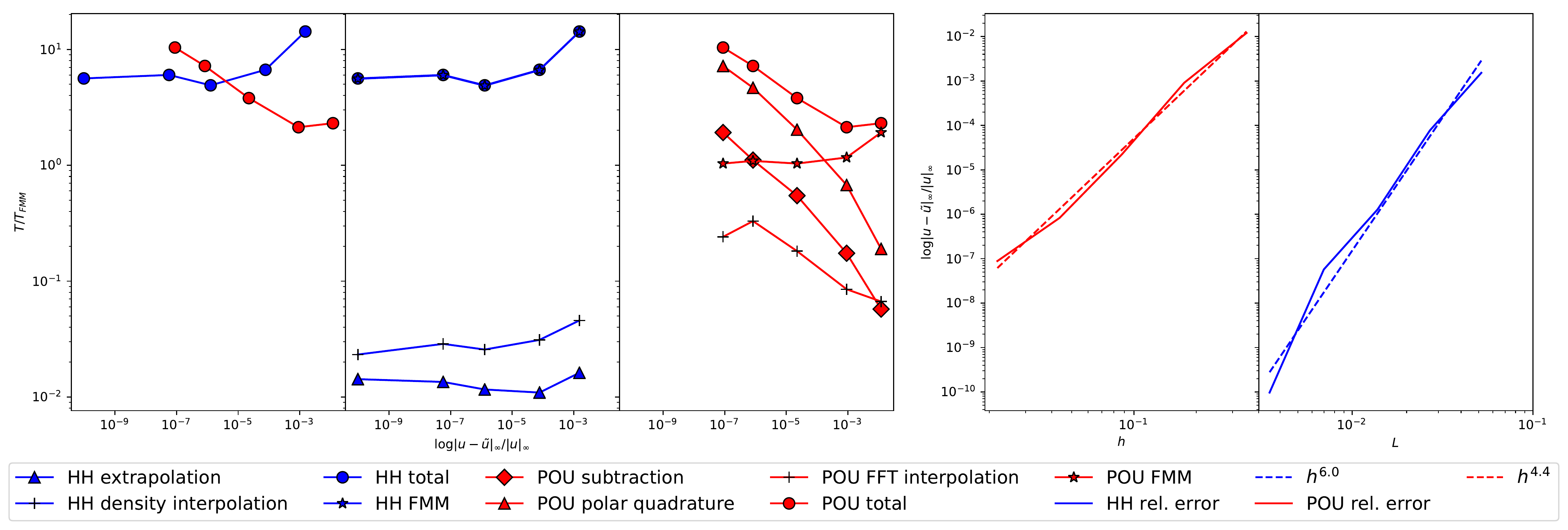} }
    \makebox[\textwidth][c]{ \includegraphics[width=1.1\textwidth]{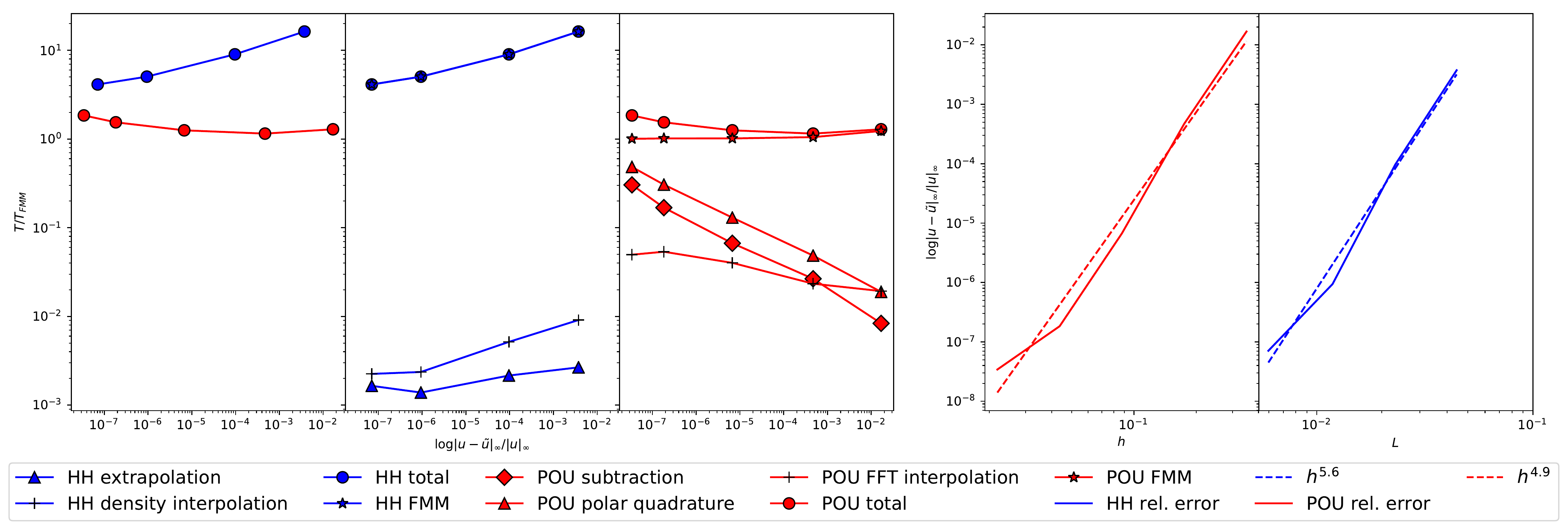} }
    \mcaption{fig:compare-solve-surface}{Comparison of \qbkix on polynomial patches (HH) versus \cite{YBZ} on the surface representation of \cite{ying2004simple} (POU) solving via \gmres for $u_c$}{
    Laplace (top) and elasticity (bottom) problems solved on the spheroid shown in \cref{fig:greens-id-test-cases}.
From left to right, we plot the total cost of each scheme, the cost of each subroutine for \qbkix (blue) and the singular quadrature scheme of \cite{YBZ} (red), and the relative error as a function of $h$.
    We plot error convergence of \cite{YBZ} as a function of $h$ and \qbkix as a function of $L$, due to the distinct discretizations.
    For \qbkix parameters, we choose $r=.013\sqrt{L}$, $R=.075\sqrt{L}$ for the Laplace problem; for the elasticity problem, we choose $r=.013\sqrt{L}$, $R=.08\sqrt{L}$. We choose $p=6$ and $q=15$ for both problems.
    For \cite{YBZ} the spacing is $h_0=.35$.
    Note that in the \qbkix timing breakdown, since the \fmm time is dominant, the \fmm cost lies directly on top of the total cost.
  }
\end{figure}
The results are shown in \cref{fig:compare-solve-surface}. 
From left to right, each plot details the total cost of each scheme, the cost of each subroutine for \qbkix (denoted \abbrev{HH}) and the singular quadrature scheme of \cite{YBZ} (denoted \abbrev{POU}), and the relative error as a function of $h$ and $L$, respectively, 
for all refinement levels. 
We plot the cost of both schemes the cost of each algorithmic step as a function of their computed relative error. 
In each figure, we present results for a Laplace problem (top) and an elasticity problem (bottom). 

In \cref{fig:compare-solve-surface}, as expected, we observe a  higher  convergence rate for \qbkix compared to \cite{YBZ}. 
\cite{YBZ} outperforms \qbkix in terms of cost for all tested discretizations.  
 We observe that the \fmm evaluation in \cref{fig:compare-solve-surface} accounts for at least 95\% of the \qbkix cost.
This means that a local singular quadrature method (based on corrections to an \fmm evaluation, \cref{sec:related_work}) of \textit{worse} complexity can beat a global method, simply by virtue of reducing the \fmm size.
By noting the large difference between the \qbkix \fmm cost and the \qbkix density interpolation, we can reasonably infer that a local \qbkix scheme should narrow this  performance gap and outperform \cite{YBZ} for larger problems, assuming that switching to a local scheme does not dramatically affect error convergence. 


\subsection{Requested target precision vs. computed accuracy}
\begin{figure}[!htb]
\begin{tikzpicture}
    \begin{loglogaxis}[
        xlabel=$\err{target}$,
        ylabel= $\infty$-norm relative error,
        ylabel near ticks,
        xlabel near ticks,
        label style={font=\scriptsize},
        tick label style={font=\scriptsize},
        width=.333\linewidth
    ]
    \addplot[smooth,mark=*,blue] plot coordinates {
        (1e-4,0.0003143683435301948)
        (1e-5,6.031924968205908e-05)
        (1e-6,4.103514547039301e-06)
        (1e-7,6.614064613404895e-07)
        (1e-8,1.122384348756594e-07)
    };
        \addplot[domain=1e-8:1e-4,dotted,thick]{x};

    \end{loglogaxis}
    \end{tikzpicture}
    \begin{tikzpicture}
    \begin{semilogxaxis}[
        xlabel=$\err{target}$,
        ylabel= Target points/second/core,
        ylabel near ticks,
        xlabel near ticks,
        label style={font=\scriptsize},
        tick label style={font=\scriptsize},
        width=.333\linewidth
    ]
    \addplot[smooth,mark=*,blue] plot coordinates {
        (1e-4,2077)
        (1e-5,1620)
        (1e-6,772)
        (1e-7,400)
        (1e-8,222)
    };
    \end{semilogxaxis}
    \end{tikzpicture}
    \begin{tikzpicture}
    \begin{loglogaxis}[
        xlabel=$\err{target}$,
        ylabel= Number of patches,
        ylabel near ticks,
        xlabel near ticks,
        label style={font=\scriptsize},
        tick label style={font=\scriptsize},
        legend style={font=\tiny},
         legend image post style={scale=0.5},
        width=.333\linewidth
    ]
    \addplot[smooth,mark=*,blue] plot coordinates {
        (1e-4,128)
        (1e-5,128)
        (1e-6,128)
        (1e-7,128)
        (1e-8,128)
    };
    \addlegendentry{$|\Pcoarse|$}
    \addplot[smooth,mark=*,red] plot coordinates {
        (1e-4,5408)
        (1e-5,8288)
        (1e-6,28256)
        (1e-7,60032)
        (1e-8,113936)
    };
    \addlegendentry{$|\Pfine|$}
    \end{loglogaxis}
    \end{tikzpicture}
        \mcaption{fig:full-algo-perf}{Performance of full algorithm}{
            Left: $\infty$-norm relative error in singular integral vs requested target accuracy (blue). 
        The dotted line is the ideal behavior $y=x$.
        Middle: Performance in terms of target points evaluated per second per core with \qbkix.
        Right: Number of patches in $\Pcoarse$ and $\Pfine$ computed by the preprocessing algorithms.}
\end{figure}

In this section, we study the performance of the full algorithm outlined in \cref{sec:algo}.
We test \qbkix on the torus domain shown in \cref{fig:greens-id-test-cases}-right.
We choose a reference solution of the form of \cref{eq:point-charge-solution} with a single point charge located at the origin, in the middle of the hole of the torus.
We solve the integral equation with two-sided \qbkix and evaluate the singular integral on a distinct discretization with one-sided \qbkix.
We choose $q=20$, $p=6$ and $a=b/6$.
We select various values for $\err{target}$ using the plot in \cref{fig:extrap-err-p6} to choose $b$ to ensure sufficiently accurate extrapolation. 
We plot the results of our tests in \cref{fig:full-algo-perf}.

We see in \cref{fig:full-algo-perf}-left that we are consistently close to the requested target precision. 
We see a decline in target points per second per core as accuracy increases in \cref{fig:full-algo-perf}-middle.
This is explained by \cref{fig:full-algo-perf}-right, which shows an increase in the size $\Pfine$ as $\Pcoarse$ remains a fixed size.
The initial 128 patches in $\Pcoarse$ are enough to resolve the boundary condition and $\Gamma$, but we need greater quadrature accuracy for lower values of $\err{target}$ .
Decreasing the number of points in passed to the \fmm, i.e., decreasing the size of $\Pfine$, is the main way to improve performance of our method.
This is further indication that a local version of \qbkix will outperform a global approach.

\subsection{Full algorithm on interlocking torii\label{sec:results-torii}}
We now demonstrate the full algorithm pipeline on an exterior Laplace problem, whose boundary is defined by four interlocking torii shown in \cref{fig:torii}.
The domain boundary is contained in the box $[-3.8, 2.4] \times [-1.1, 1.1] \times [-1,1]$.
The shortest distance between two adjacent torii is less than 10\% of a polynomial patch length defining the boundary.
We again use a boundary condition of the form \cref{eq:point-charge-solution} with a single point charge located at $(0,.03,.875)$, inside the upper half of the second torus from the right in \cref{fig:torii}.
This problem is challenging due to the nearly touching geometry of the torii, along with the singularity placed close to the boundary.
We run the admissibility and adaptive upsampling algorithms outlined in \cref{sec:algo}, solve \cref{eq:int-eq} using two-sided \qbkix, and evaluate the solution on the boundary using one-sided \qbkix.
The absolute error in the $\infty$-norm of the singular evaluation is plotted on the boundary surface.

\begin{figure}[!htb]
  \centering
  \setlength\figureheight{1.9in}
  \setlength\figurewidth{2.1in}
  \hfill
  \begin{minipage}{.65\textwidth}
    \includegraphics[width=\linewidth]{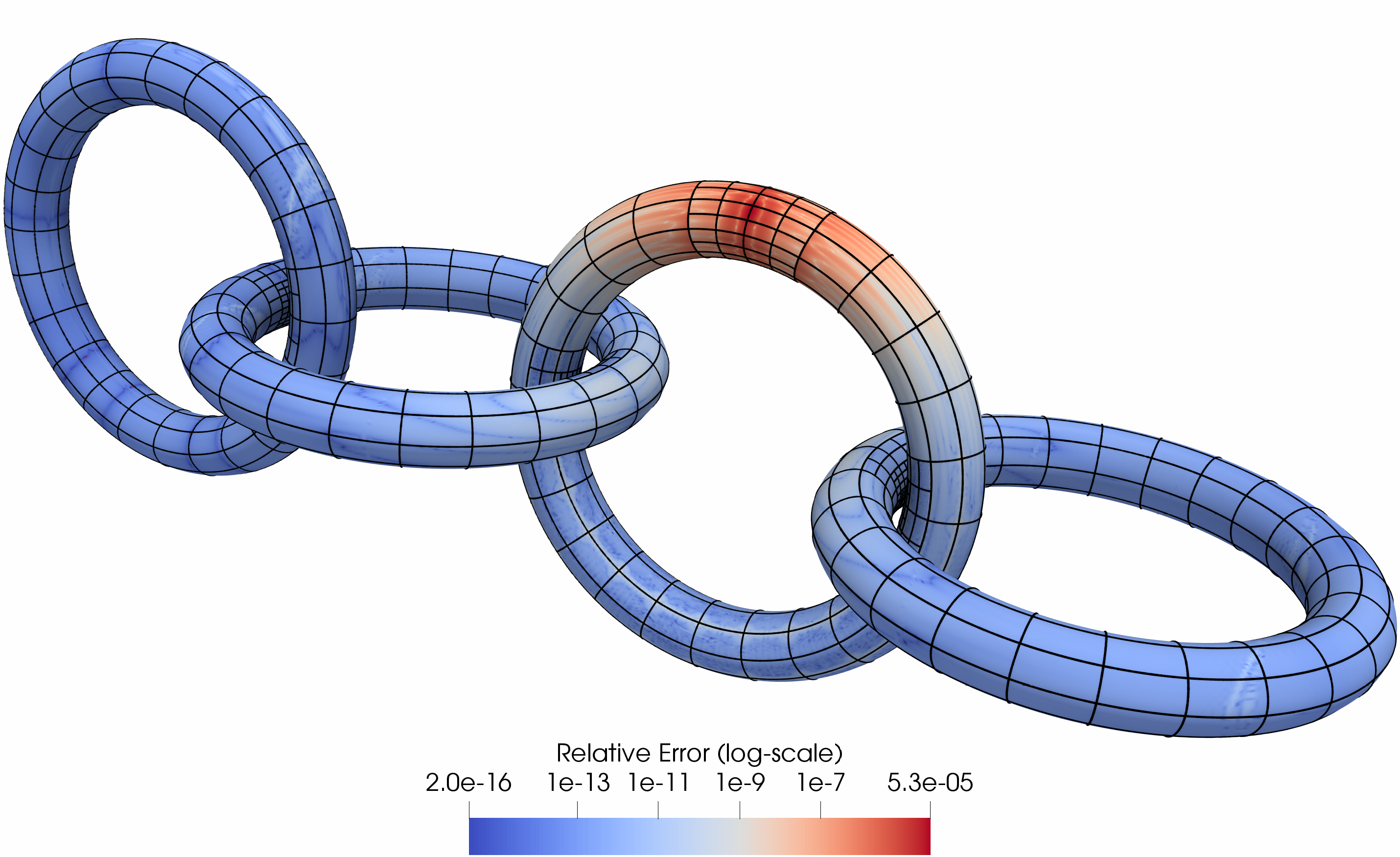}
  \end{minipage}\hfill
  \begin{minipage}{.35\textwidth}
    \includegraphics[width=\linewidth]{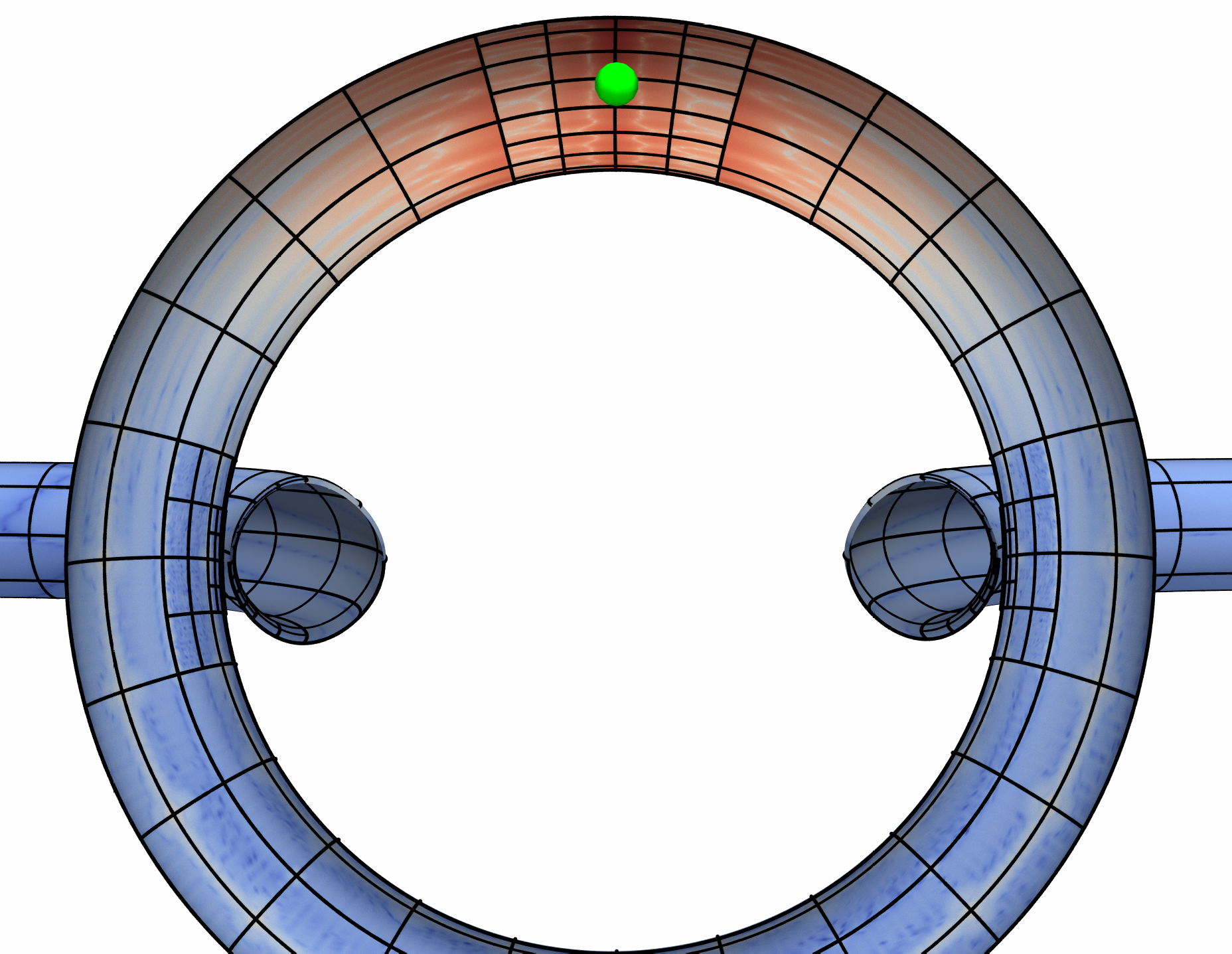}
  \end{minipage}\hfill
  \mcaption{fig:torii}{Absolute error of \gmres solve via \qbkix on interlocking torii}{Left: The admissible set of 1128 patches in $\Pcoarse$ used to solve \cref{eq:int-eq} is shown (black lines denote patch boundaries).
  The point charge generated the boundary condition is located within the second torus from the right. 
    Right: a cross-section of the torii geometry through the $xz$-plane, showing the second torus from the right and the location of the singularity (green point).
}
\end{figure}
Using $a=.1$, $b=.025$, $p=6$ and $q=20$, we achieve a maximum pointwise error of $1.29\times 10^{-5}$. 
\gmres was able to reduce the residual by a factor of $10^{-13}$ over 109 iterations.
There are 288768 quadrature points in the coarse discretization, 18235392 quadrature points in the fine discretization, and 3465216 check points used in the two-sided \qbkix evaluation inside \gmres.
We evaluate the solved density at 451200 points on the boundary with one-sided \qbkix to produce the render in \cref{fig:torii}.
On a machine with two Intel Xeon E-2690v2 3.0GHz \abbrev{CPU}'s, each with 10 cores, and 100 \abbrev{GB} of \abbrev{RAM},
the \gmres solve and interior evaluation required 5.7 hours and can evaluate the singular integral at a rate of 1709 target points per second per core.

\subsection{Solution on complex geometry\label{sec:results-blood-vessel}}

\begin{figure}[!htb]
  \centering
  \setlength\figureheight{1.9in}
  \setlength\figurewidth{2.1in}
  \hfill
  \begin{minipage}{.7\textwidth}
    \includegraphics[width=\linewidth]{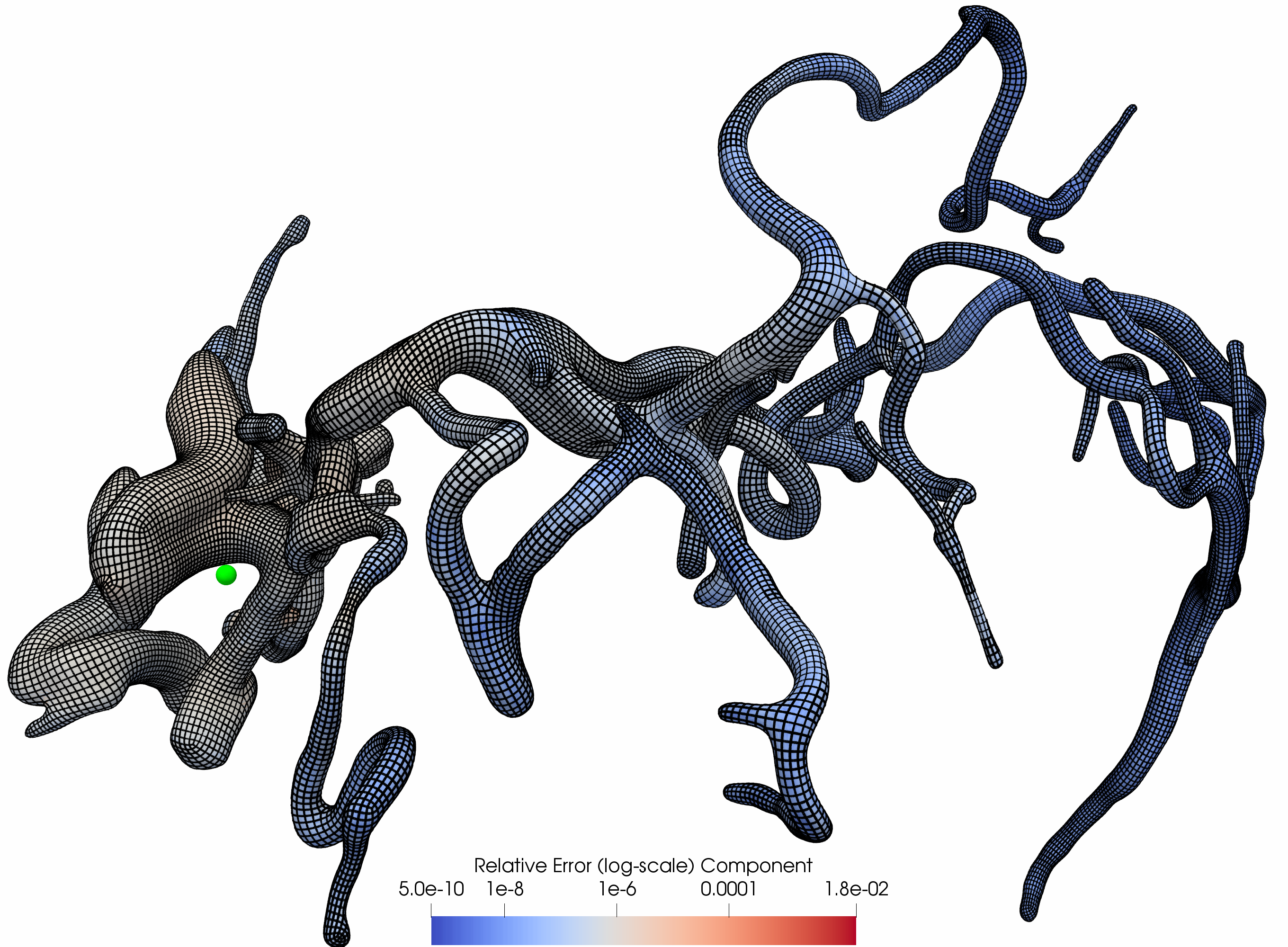}
  \end{minipage}\hfill
  \mcaption{fig:vessel}{Absolute error of \gmres solve via \qbkix on complex blood vessel geometry used in \cite{lu2019scalable}}{The blood vessel uses 40,960 8th order polynomial patches (black edges denote patch boundaries). The geometry is admissible by construction. 
    The point charge is located on left side of the figure (green)}
\end{figure}

We have demonstrated in \cite{lu2019scalable} a parallel implementation of \cref{sec:singular-eval}, applied to simulating red blood cell flows.
The surface geometry of the blood vessel shown in \cref{fig:vessel} is complex, with rapidly varying curvatures and geometric distortions due to singular vertices in the surface mesh.
Since the surface is admissible, we are able to apply parallel \qbkix directly without geometric preprocessing to solve an interior Dirichlet Stokes problem.
We use $a=.125$, $b=.125$, $p=6$ and $q=16$ as simulation parameters.

Using 32 machines each with twenty 2.6 Ghz cores with 100\abbrev{GB} of \abbrev{RAM}, we achieve a maximum pointwise error of $3\times 10^{-6}$ when solving a Stokes problem with constant density. 
We then place a random vector point charge two patch lengths away (relative to the patches in $\Pcoarse$) from the domain boundary (on the left side of \cref{fig:vessel}, solve \cref{eq:int-eq} using two-sided \qbkix, and evaluate the solution on the boundary using one-sided \qbkix.
The absolute error in the $\infty$-norm of the singular evaluation is plotted on the boundary surface.
There are 10,485,760 quadrature points in the coarse discretization, 167,772,160 quadrature points in the fine discretization, and 125,829,120 check points used in the two-sided \qbkix evaluation inside \gmres.
We evaluate the solved density at 209,715,200 points on the boundary with one-sided \qbkix to produce the render in \cref{fig:torii}.
We achieve a maximum pointwise error of $1.8\times 10^{-2}$ and can evaluate the singular integral at rate of 3529 target points per second per core. 

%% file: conclusion.tex
\section{Conclusion\label{sec:conclusion}}
We have presented \qbkix, a fast, high-order, kernel-independent, singular/near-singular quadrature scheme for elliptic boundary value problems in \threed on complex geometries defined by piecewise tensor-product polynomial surfaces.
The primary advantage of our approach is \textit{algorithmic simplicity}: the algorithm can implemented easily with an existing smooth quadrature rule, a point \fmm and \oned and \twod interpolation schemes.
We presented fast geometry processing algorithms to guarantee accurate singular/near-singular integration, adaptively upsample the discretization and query local surface patches.
We then evaluated \qbkix in various test cases, for Laplace, Stokes, and elasticity problems on various patch-based geometries and compared our approach with \cite{YBZ}.

\cite{lu2019scalable} demonstrates a parallel implementation of \qbkix, but the geometric preprocessing and adaptive upsampling algorithms presented in \cref{sec:algo} are not parallelized.
This is a requirement to solve truly large-scale problems that exist in engineering applications.
Our method can also be easily restructured as a local method.
The comparison in \cref{sec:results-compare} highlights an important point: a local singular quadrature method can outperform a global method for moderate accuracies, \textit{even when the local scheme is asymptotically slower}.
This simple change can also dramatically improve both the serial performance and the parallel scalability of \qbkix shown in \cite{lu2019scalable}, due to the decreased communication of a smaller parallel \fmm evaluation.
The most important improvement to be made, however, is the equispaced extrapolation.
Constructing a superior extrapolation procedure, optimized for the boundary integral context, is the main focus of our current investigations.

\section{Acknowledgements}
We would like to thank Michael O'Neil, Dhairya Malhotra, Libin Lu, Alex Barnett, Leslie Greengard, Michael Shelley for insightful conversations, feedback and suggestions regarding this work. 
We would also like to thank the NYU HPC team, and Shenglong Wang in particular, for great support throughout the course of this work, and the helpful feedback of the anonymous reviewers.
This work was supported by NSF grant DMS-1821334.

%% file: appendix.tex
\appendix

\section{Derivation of \Cref{heuristic:error_quad_high_order} \label{app:proof_of_error_quad_high_order}}
We are interested in computing the error incurred when approximating a \twod surface integral with an interpolatory quadrature rule. 
In \oned on the interval $[-1,1]$, we're interested in the quantity
\begin{align}
  R_q[f] &= I[f] - Q_q[f] \label{eq:rem1d}\\
  \shortintertext{where}
  I[f] &= \int_{-1}^1 f(x) dx\\
  Q_q[f] &= \sum_{i=0}^q f(x_i)w_i,\\
\end{align}
for quadrature weights $w_i$ for a $q$-point quadrature rule.
For a \twod double integral, we define a similar relationship between the remainder, the exact integral and the $q$th order quadrature rule:
\begin{align}
  R^{(2)}_q[f] &= I^{(2)}[f] - Q^{(2)}_q[f]\label{eq:rem2d}\\
  \shortintertext{where}
    I^{(2)}[f] &= \int_{-1}^1\int_{-1}^1 f(s,t) dsdt\\
  Q^{(2)}_q[f] &= \sum_{j=0}^q\sum_{i=0}^q f(s_i,t_j)w_iw_j,
\end{align}

For a function of two variables $f(s,t)$, we will denote $I_s[f] = \int_{-1}^1f(s,\cdot)ds$ as integration with respect to the $s$ variable only, which produces a function of $t$.
The same subscript notation applies to $R_{q,s}[f]$ and $Q_{q,s}[f]$ and use similar notation  for $t$: we apply the \oned functional to the variable in the subscript, producing a \oned function in the remaining variable.
We observe that 
\begin{equation}
  I^{(2)}[f] = \int_{-1}^1\left(\int_{-1}^1 f(s,t) ds\right)dt = \int_{-1}^1 I_s[f] dt = I_t[I_s[f]]
  \label{eq:rem2d_to_1d}
\end{equation}
Following the discussion in \cite{aT2}, we substitute into \cref{eq:rem2d_to_1d} and have
\begin{align}
  I^{(2)}[f] &= I_t[R_{q,s}[f] + Q_{q,s}[f]]\\
         &= R_{q,t}[R_{q,s}[f] + Q_{q,s}[f]] + Q_{q,t}[R_{q,s}[f] + Q_{q,s}[f]] \\
         &= R_{q,t}[R_{q,s}[f]] + Q_{q,s}[R_{q,t}[f]] + Q_{q,t}[R_{q,s}[f]] +Q_{q,t}[Q_{q,s}[f]]
\end{align}
We assume that the higher-order ``remainder of remainder'' term contributes negligibly to the error.
Although it has been shown that this term has a non-trivial contribution to a tight error estimate \cite{elliott2015complete}, 
we are able to provide a sufficiently tight upper bound.
For large $q$, the quadrature rule approaches the value of the integral, i.e., $Q_{q,\beta} \approx I_\beta$ for $\beta =s,t$, we're left with:
\begin{align}
  I^{(2)}[f] &\approx  I_s[R_{q,t}[f]] + I_t[R_{q,s}[f]] +Q^{(2)}_q[f],\\
  \shortintertext{and hence:}
  R^{(2)}_q[f] &\lesssim  I_s[R_{q,t}[f]] + I_t[R_{q,s}[f]] \label{eq:rem2d_int},
\end{align}
where $\lesssim$ means "approximately less than or equal to."
From \cite[Theorem 5.1]{trefethen2008gauss}, we recall that for a \oned function $\theta$ defined on $[-1,1]$, if $Q_q[\theta]$ is computed with Clenshaw-Curtis quadrature, $\theta$ is $C^k$ and $\|\theta^{(k)}\|_T < V$ on $[-1,1]$ for real finite $V$, then for sufficiently large $q$, the following inequality holds
\begin{equation}
  R_q[\theta] \leq \frac{32V}{15\pi k(2q+1-k)^k},
  \label{eq:tref_ccquad2d}
\end{equation}
where $\|\alpha(x)\|_T = \|\alpha^\prime/\sqrt{1-x^2}\|_1$.
We're interested in integrating a function $\tilde{\theta}$ over an interval $[-h,h]$ for various $h$.
If $\tilde{\theta}$ is $C^k$ and $\|\tilde{\theta}\|_T < V^\prime$ on $[-h,h]$ for a real constant $V^\prime$ independent of $h$, then we can define $\theta(x) = \tilde{\theta}(hx)$ on $[-1,1]$ and apply 
\cref{eq:tref_ccquad2d}:
\begin{equation}
  R_q[\tilde{\theta}] \leq \frac{32h^{k+1}V^\prime}{15\pi k(2q+1-k)^k}.
  \label{eq:tref_ccquad2d_variable_interval}
\end{equation}
This follows directly from the proof of \cite[Theorem 4.2]{trefethen2008gauss} applied to $\theta$ by replacing $\theta$ with $\tilde{\theta}(hx)$ and noting that $\theta^{(k)}(x) = h^k \tilde{\theta}^{(k)}(hx)$.
The change of variables produces the first power of $h$, while each of the $k$ integration by parts produces an additional power of $h$.
In the context of \qbkix, the size of $h$ is proportional to the edge length of the subdomain $D_i$ outlined in \cref{sec:geom-def}.

Applying \cref{eq:tref_ccquad2d_variable_interval} to \cref{eq:rem2d_int}, and again letting $f(s,t) = \Theta(hs,ht)$, gives us
\begin{equation}
  R^{(2)}_q[f] \lesssim \frac{32h^{k+1}}{15\pi k (2q+1-k)^k}\left[I_s[V^\prime_t(s)] + I_t[V^\prime_s(t)]\right] \label{eq:quad_err_1d}
\end{equation}
where $V^\prime_t(s)=\max_t\|\Theta^{(k)}(hs, ht)\|_T$ and $V^\prime_s(t)=\max_s\|\Theta^{(k)}(hs, ht)\|_T$ for fixed values of $s,t$.
If we can choose a $\tilde{V}$ that is strictly greater than $V^\prime_s(t)$ and $V^\prime_t(s)$ for any $s,t$ in $\mathcal{I}^{(2)}$, we are left with
\begin{equation}
  R^{(2)}_q[f] \lesssim \frac{128h^{k+1}\tilde{V}}{15\pi k (2q+1-k)^k}.\label{eq:quad_err_final}
\end{equation}

Applying this to the integration of double layer potentials, we can simply let $\tilde{V}$ be the largest variation of the $k$th partial derivatives of the integrand of any single patch in \cref{eq:double_layer_patches}.
In fact, we know that this value is achieved at the projection of $\vx$ on the patch $P_i$ closest to $\vx$, i.e., $(s^*,t^*) = \mathrm{argmin}_{\mathcal{I}^{(2)}} \|\vx - P_i(s,t)\|_2$. 
 We can also choose $h = \max_i h_i$ to observe standard high-order convergence as a function of patch domain size, which we summarize in the following theorem.
  The smoothness and bounded variation assumptions required to apply \cref{eq:tref_ccquad2d} to our layer potential follow directly from the smoothness of $u(\vx)$ in $\Omega$.
  Our heuristic directly follows.